\documentclass[11pt,reqno]{amsart}

\usepackage{amsmath,amstext,amssymb,amscd}
\usepackage{verbatim}
\usepackage{enumerate}
\usepackage{mathrsfs}
\usepackage[dvipsnames,usenames]{xcolor}

\usepackage{amsthm}

\newtheorem{theorem}{Theorem}[section]
\newtheorem{lemma}[theorem]{Lemma}

\theoremstyle{definition}
\newtheorem{definition}[theorem]{Definition}

\theoremstyle{remark}

\numberwithin{equation}{section}
\begin{document}

\def\bar{\overline}

\author[C. Cao]{Chongsheng Cao}
\address{Department of Mathematics \& Statistics \\Florida International University\\
Miami, Florida 33199, USA} \email{caoc@fiu.edu}

\author[Y. Guo]{Yanqiu Guo}
\address{Department of Mathematics \& Statistics \\Florida International University\\
Miami, Florida 33199, USA} \email{yanguo@fiu.edu}

\author[E. S. Titi]{Edriss S. Titi}
\address{Department of
Mathematics \\ Texas A\&M University\\
College Station, TX 77843, USA;
Department of Applied Mathematics and Theoretical Physics, University of Cambridge, Cambridge CB3 0WA, UK;
\textbf{AND} Department of Computer Science and Applied Mathematics \\ Weizmann Institute of Science \\ Rehovot 7610001 Israel}
 \email{titi@math.tamu.edu,  edriss.titi@damtp.cam.ac.uk}

\keywords{Rayleigh-B\'enard convection, rotational fluid flow, cyclonic and anticyclonic coherent structures,
incompressible, infinite Prandtl number, global well-posedness, weak solutions, strong solutions}
\subjclass[2010]{35A01, 35A02, 35Q35, 35K55}

\title[Rapidly rotating convection in the limit of infinite Prandtl number]
{Global well-posedness for a rapidly rotating convection model of tall columnar structure in the limit of infinite Prandtl number}
\date{March 7, 2020}
\maketitle

\begin{abstract}
We analyze a three-dimensional rapidly rotating convection model of tall columnar structure in the limit of infinite Prandtl number, i.e., when
the momentum diffusivity is much more dominant than the thermal diffusivity.
Consequently, the dynamics of the velocity field takes place at a much faster time scale than the temperature fluctuation, and at the limit the velocity field formally adjusts instantaneously to the thermal fluctuation. We prove the global well-posedness of weak solutions and strong solutions to this model.
\end{abstract}

\bigskip
\begin{center}
{\it Dedicated to Professor Matthias Hieber on the occasion of his 60th birthday}
\end{center}
\bigskip

\section{Introduction}
In geophysics, thermal convection is often influenced by planetary rotation. Coherence structures in convection under moderate rotation are exclusively cyclonic. However, for rapid rotation, experiments have revealed a transition to equal populations of cyclonic and anticyclonic structures. For instance, the flow visualization experiments of Vorobieff and Ecke \cite{Vorobieff} identified a striking topological change in the dynamics of the vortices. In the strongly nonlinear and turbulent regimes, plume generation in the thermal boundary layer results in a new population of anticyclonic plumes, in addition to the cyclonic population. Also, the distribution of cyclonic and anticyclonic coherent structures approaches a balance as the Rossby number approaches zero. In order to study such interesting phenomenon numerically, Sprague et al. \cite{SJKW}
derived a reduced system of equations for rotationally constrained convection valid in the asymptotic limit of thin columnar structures and rapid rotation. Performing a numerical simulation of Rayleigh-B\'enard convection in an infinite layer rotating uniformly about the vertical axis, visualization indicates the existence of cyclonic and anticyclonic vortical population in \cite{SJKW}, which is consistent with the experimental results described in \cite{Vorobieff}. Also see \cite{JKMW, JK}.

The reduced three-dimensional rapidly rotating convection model of tall columnar structure introduced in \cite{SJKW} is given by the system:
\begin{align}
&\frac{\partial w}{\partial t} + \mathbf u \cdot \nabla_h w + \frac{\partial \psi}{\partial z} = \frac{\text{Ra}}{\text{Pr}}\theta' + \Delta_h w,   \label{m-1} \\
&\frac{\partial \omega}{\partial t} + \mathbf u \cdot \nabla_h \omega - \frac{\partial w}{\partial z} = \Delta_h \omega,   \label{m-2}\\
&\frac{\partial \theta'}{\partial t}  + \mathbf u \cdot \nabla_h \theta' +  w \frac{\partial \bar{\theta}}{\partial z} = \frac{1}{\text{Pr}} \Delta_h \theta',   \label{m-3}\\
&\frac{\partial(\bar{\theta' w})}{\partial z} = \frac{1}{\text{Pr}} \frac{\partial^2 \bar \theta}{\partial z^2},       \label{m-4}\\
& \nabla_h \cdot \mathbf u=0.     \label{m-5}
\end{align}
The above system is considered subject to periodic boundary conditions in $\mathbb R^3$ with fundamental periodic domain $\Omega = [0,2\pi L]^2 \times [0,2\pi]$. The unknowns are functions of $(x,y,z,t)$, where $(x,y,z)\in \Omega$ and $t\geq 0$. In (\ref{m-1})-(\ref{m-5}), $\nabla_h=(\frac{\partial}{\partial x}, \frac{\partial}{\partial y})$ denotes the horizontal gradient and $\Delta_h=\frac{\partial^2}{\partial x^2}+ \frac{\partial^2}{\partial y^2}$ denotes the horizontal Laplacian. In the model, $\mathbf u=(u,v)$ is the horizontal component of the three-dimensional velocity vector field $(u,v,w)$, and
$\omega=  \nabla_h \times \mathbf u=\partial_x v - \partial_y u$ denotes the vertical component of the vorticity. Moreover, the stream function for the horizontal flow is denoted by $\psi=\Delta_h ^{-1}\omega$ such that its horizontal average $\bar{\psi} = \frac{1}{4 \pi^2 L^2}  \int_{[0,2\pi L]^2} \psi(x,y,z,t) dx dy = 0$. In addition,
$\theta'=\theta - \bar{\theta}$ represents the horizontal fluctuation of the temperature $\theta$, where
$\bar{\theta}(z,t)=\frac{1}{4\pi^2 L^2} \int_{[0,2\pi L]^2} \theta(x,y,z,t) dx dy$ is the horizontal-mean temperature. In the above system, Ra is the Rayleigh number and Pr is the Prandtl number. We comment that the assumption
$\nabla_h \cdot \mathbf u=0$ means that the \emph{horizontal} flow is divergence-free.

The global regularity for system (\ref{m-1})-(\ref{m-5}) is unknown. The main difficulty of analyzing (\ref{m-1})-(\ref{m-5}) lies in the fact that the physical domain is three-dimensional, whereas the regularizing viscosity acts only on the horizontal variables, and the equations contain troublesome terms $\frac{\partial \phi}{\partial z}$ and $\frac{\partial w}{\partial z}$ involving the derivative in the vertical direction. In our recent paper \cite{CGT-2}, system (\ref{m-1})-(\ref{m-5}) was regularized by a weak dissipation term and the global well-posedness of strong solutions was established for the regularized system.

System (\ref{m-1})-(\ref{m-5}) is a reduced model derived from the three-dimensional Boussinesq equations by using the asymptotic theory. Generally speaking, the Boussinesq approximation for buoyancy-driven flow is applied to problems where the fluid varies in temperature from one place to another, driving a flow of fluid and heat transfer. In particular, the Boussinesq approximation to the Rayleigh-B\'enard convection is a system of equations coupling the three-dimensional Navier-Stokes equations to a heat advection-diffusion equation. For small Rossby number (i.e., rapid rotation) and large ratio of the depth of the fluid layer to the horizontal scale (i.e., tall columnar structures), the 3D Boussinesq equations under the influence of a Coriolis force term, can be reduced to  system (\ref{m-1})-(\ref{m-5}) asymptotically.
The derivation of model (\ref{m-1})-(\ref{m-5}) was motivated by the Taylor-Proudman constraint \cite{Proudman, Taylor} which suggests that rapidly rotating convection takes place in tall columnar structures.

According to the derivation of model (\ref{m-1})-(\ref{m-5}) in \cite {SJKW} from the 3D Boussinesq equations,
the state variables, i.e., the velocity, pressure and temperature, are expanded in terms of the small parameter $Ro$, which stands for the Rossby number. For rapidly rotating flow, i.e., $Ro \ll 1$, the leading-order flow is horizontally divergence-free (see \cite{SJKW}). Roughly speaking, if the Rossby number is small in the Boussinesq equations, then the Coriolis force and the pressure gradient force are relatively large, which results in an equation in which the leading order terms are in geostrophic balance: the pressure gradient force is balanced by the Coriolis effect. Then taking the curl of the geostrophic balance equation implies that the \emph{horizontal} flow is divergence-free, namely, $\nabla_h \cdot \mathbf u=u_x + v_y=0$. In addition, the term $\frac{\partial \psi}{\partial z}$ in equation (\ref{m-1}) also originates from the geostrophic balance.

Since the original Boussinesq equations are considered in a tall column
$(x,y,\tilde z) \in [0,2\pi L]^2 \times [0,2\pi \ell]$, where the aspect ratio $\ell/L \gg 1$,
it is natural to introduce the scaled vertical variable $z=\tilde z/\ell \in [0,2\pi]$ which appears in system (\ref{m-1})-(\ref{m-5}). Set $u(x,y,z)=\tilde u(x,y,\tilde z)$, $v(x,y,z)=\tilde v(x,y,\tilde z)$ and $w(x,y,z)=\tilde w(x,y,\tilde z)$, where $(\tilde u, \tilde v, \tilde w)^{tr}$ represents the velocity vector field for the original Boussinesq equations.
Note, the divergence-free condition of $(\tilde u, \tilde v, \tilde w)^{tr}$ reads
$\frac{\partial \tilde u}{\partial x} + \frac{\partial \tilde v}{\partial y}  + \frac{\partial \tilde w}{\partial \tilde z} =0$, which implies that $\frac{\partial u}{\partial x} + \frac{\partial v}{\partial y} + \frac{1}{\ell} \frac{\partial w}{\partial z}=0$.
Then, for $\ell \gg 1$, one can ignore the term $\frac{1}{\ell}  \frac{\partial w}{\partial z}$ and obtain that $\frac{\partial u}{\partial x} + \frac{\partial v}{\partial y}=0$. In sum, the fast rotation and the tall columnar structure both imply that to leading order terms the horizontal flow is divergence-free. Moreover, the absence of vertical diffusion in system (\ref{m-1})-(\ref{m-5}) is also a consequence of the large aspect ratio of the fluid region. Furthermore, it is remarked in \cite{SJKW} that in the classical small-aspect-ratio (flat) regime, the strong stable stratification permits weak vertical motions only, while in the present large-aspect-ratio (tall) case, the unstable stratification permits substantial vertical motions.

There are two dimensionless numbers in system (\ref{m-1})-(\ref{m-5}). They are the Prandtl number Pr and the Rayleigh number Ra. The Prandtl number Pr represents the ratio of molecular diffusion of momentum to molecular diffusion of heat. More precisely, one defines
\begin{align}
\text{Pr}=\frac{\text{momentum diffusivity}}{\text{thermal diffusivity}} = \frac{\nu}{\alpha}= \frac{\mu/\rho}{k/(c_p \rho)}
=\frac{c_p \mu}{k}.
\end{align}
Here, $\nu$ represents momentum diffusivity, i.e., kinematic viscosity. Notice that $\nu=\mu/\rho$, where $\mu$ is dynamics viscosity and $\rho$ is the constant density.  Also, $\alpha$ stands for the thermal diffusivity,
which is equal to $k/(c_p \rho)$, where $k$ is thermal conductivity and $c_p$ represents the specific heat capacity of the fluid. Fluids with small Prandtl numbers are free-flowing liquids with high thermal conductivity and are therefore a good choice for heat transfer liquids. Liquid metals such as mercury have small Prandtl numbers. On the other hand, with increasing viscosity, the Prandtl number also increases, leads to the phenomenon that the momentum transport dominates over the heat transport and acts on a faster time scale. For instance, concerning the engine oil, convection is very effective in transferring energy in comparison to pure conduction, so momentum diffusivity is dominant. Another example is Earth's mantle, which has extremely large Prandtl number.

The Rayleigh number Ra is another dimensionless number appeared in model (\ref{m-1})-(\ref{m-5}).
It represents the strength of the buoyancy in the fluid driven by the heat gradient.

In this manuscript, we consider rapidly rotating convection in the limit of infinite Prandtl number Pr.
Under such scenario, the dynamics of the momentum acts on a much faster time scale than the heat dynamics.
For this problem, the appropriate time scale is the horizontal thermal diffusion time. Using the substitution (cf. \cite{SJKW})
\begin{align*}
t \rightarrow \text{Pr}\; t, \;\;\;  \mathbf u \rightarrow \frac{1}{\text{Pr}} \mathbf u, \;\;\;
w \rightarrow \frac{1}{\text{Pr}} w,
\end{align*}
system (\ref{m-1})-(\ref{m-5}) becomes
\begin{align}
&\frac{1}{\text{Pr}}\left(\frac{\partial w}{\partial t} + \mathbf u \cdot \nabla_h w\right) + \frac{\partial \psi}{\partial z} = \text{Ra}\, \theta' + \Delta_h w,   \label{m-11} \\
&\frac{1}{\text{Pr}}\left(\frac{\partial \omega}{\partial t} + \mathbf u \cdot \nabla_h \omega \right) - \frac{\partial w}{\partial z} = \Delta_h \omega,   \label{m-22}\\
&\frac{\partial \theta'}{\partial t}  + \mathbf u \cdot \nabla_h \theta' +  w \frac{\partial \bar{\theta}}{\partial z} =  \Delta_h \theta',   \label{m-33}\\
&\frac{\partial(\bar{\theta' w})}{\partial z} = \frac{\partial^2 \bar \theta}{\partial z^2},      \label{m-44} \\
& \nabla_h \cdot \mathbf u=0.     \label{m-55}
\end{align}
Then in the limit of infinite Prandtl number, i.e., letting
$\text{Pr} \rightarrow \infty$ in system (\ref{m-11})-(\ref{m-55}), one formally obtains the following system of equations
\begin{align}
& \frac{\partial \psi}{\partial z} = \text{Ra}\,\theta' + \Delta_h w,   \label{inP-1} \\
& - \frac{\partial w}{\partial z} = \Delta_h \omega,   \label{inP-2}\\
&\frac{\partial \theta'}{\partial t}  + \mathbf u \cdot \nabla_h \theta' +  w \frac{\partial \bar{\theta}}{\partial z} = \Delta_h \theta',   \label{inP-3}\\
&\frac{\partial(\bar{\theta' w})}{\partial z} = \frac{\partial^2 \bar \theta}{\partial z^2},     \label{inP-4}\\
& \nabla_h \cdot \mathbf u=0.     \label{inP-5}
\end{align}
The system is considered subject to periodic boundary conditions in $\mathbb R^3$ with fundamental periodic domain $\Omega = [0,2\pi L]^2 \times [0,2\pi]$. Here, $\omega = \nabla_h \times \mathbf u$, $\psi = \Delta_h^{-1} \omega$ such that its horizontal average $\bar \psi=0$.
Recall that the horizontal thermal fluctuation $\theta'$ of the temperature $\theta$ is defined as
$\theta'=\theta - \bar{\theta}$, where $\bar{\theta}(z,t)=\frac{1}{4\pi^2 L^2} \int_{[0,2\pi L]^2} \theta(x,y,z,t) dx dy$ is the horizontal-mean temperature. In system (\ref{inP-1})-(\ref{inP-5}), the velocity field acting on a very fast time scale, adjusts instantaneously to the dynamics of the thermal fluctuations, demonstrated by the linear equations (\ref{inP-1})-(\ref{inP-2}).
Therefore, the initial condition is imposed on $\theta'$ only: $\theta'(0)=\theta'_0$.
The purpose of this work is to prove the global well-posedness of weak and strong solutions for system (\ref{inP-1})-(\ref{inP-5}) defined on a fundamental periodic space domain $\Omega=[0,2\pi L]^2 \times [0,2\pi]$. Also, in order to obtain the uniqueness of the temperature $\theta$, we assume that the average temperature is zero, i.e., $\int_{\Omega} \theta(x,y,z,t) dx dy dz =0$, for all $t\geq 0$.

In the literature, there were some analytical studies for the three-dimensional Boussinesq equations in the limit of infinite Prandtl number. Wang \cite{Wang} rigorously justified the infinite Prandtl number convection model as the limit of the Boussinesq equations when the Prandtl number approaches infinity (see also \cite{Wang2, Wang3}). Also, for infinite Prandtl number convection, there have been several rigorous derivation of upper bounds of the upwards heat flux, as given by the Nusselt number Nu, in terms of the forcing via the imposed temperature difference, as given by the Rayleigh number in the turbulent regime Ra $\gg 1$. For example, the work \cite{CD} by Constantin and Doering was one of the early papers in the literature for this topic. More recently, by combining the background field method and the maximal regularity in $L^{\infty}$, Otto and Seis \cite{Otto2} showed that $\text{Nu}\lessapprox  \text{Ra}^{1/3} (\log \log \text{Ra})^{1/3}$ -- an estimate that is only a double logarithm away from the supposedly optimal scaling $\text{Nu} \sim \text{Ra}^{1/3}$. See also \cite{CD2, Otto, FGMMW, Otto3, Wang4, Whitehead} and references therein.

It is worth mentioning that in \cite{CFT}, Cao, Farhat and Titi established the global regularity for an inviscid three-dimensional slow limiting ocean dynamics model, which was derived as a strong rotation limit of the rotating and stratified Boussinesq equations.

The paper is organized as follows. In section \ref{sec-results}, we state main results of the paper, i.e.,
the global well-posedness of weak solutions and strong solutions for the infinite Prandtl number convection (\ref{inP-1})-(\ref{inP-5}). In section \ref{sec-ineq}, we provide some auxiliary inequalities and some well-known identities, which will be used repeatedly in our energy estimates. In section \ref{sec-weak}, we give a detailed proof for the global well-posedness of weak solutions. Finally, section \ref{sec-strong} is devoted to the proof for the global well-posedness of strong solutions.
\vspace{0.1 in}

\section{Main results}   \label{sec-results}
In this section, we give definitions of weak solutions as well as strong solutions for system (\ref{inP-1})-(\ref{inP-5}). Then, we state the main results of the manuscript, namely, the global well-posedness of weak solutions and strong solutions for system (\ref{inP-1})-(\ref{inP-5}), subject to periodic boundary conditions on a three-dimensional fundamental periodic domain $\Omega=[0,2\pi L]^2 \times [0,2\pi]$.

\subsection{Weak solutions}
For a periodic function $f$ defined on the periodic domain $\Omega=[0,2\pi L]^2 \times [0,2\pi]$, the horizontal mean of $f$ is defined as
\begin{align}   \label{mean}
\bar f(z)=  \frac{1}{4 \pi^2 L^2}\int_{[0, 2\pi L]^2} f(x,y,z) dx dy.
\end{align}

We define the space $H^1_h(\Omega)$ of periodic functions on $\Omega$ with horizontal average zero by
\begin{align*}
H^1_h(\Omega)=\{f\in L^2(\Omega): \nabla_h f \in L^2(\Omega) \; \text{and} \; \bar{f}=0 \},
\end{align*}
with the norm $\|f\|_{H^1_h(\Omega)}= \left(\int_{\Omega}  |\nabla_h f|^2 dx dy dz\right)^{1/2}$.
We denote by $(H^1_h(\Omega))'$ the dual space of $H^1_h(\Omega)$.

For $s>0$, we denote the space of $H^s$ periodic functions on $[0,2\pi]$ with average value zero by:
\begin{align}  \label{Hdot}
\dot H^s(0,2\pi)= \{\phi \in H^s(0,2\pi): \int_0^{2\pi} \phi(z) dz =0\}.
\end{align}
Also, we denote the dual space of $\dot H^1(0,2\pi)$ by $H^{-1}(0,2\pi)= (\dot H^1(0,2\pi))'$.

Recall $\theta'=\theta- \bar \theta$ represents the fluctuation of the temperature $\theta$, about the horizontal average. Also, $(\mathbf u, w)=(u,v,w)$ is the three-dimensional velocity vector field on the periodic domain $\Omega$.

Let us define a weak solution for system (\ref{inP-1})-(\ref{inP-5}).
\begin{definition}  \label{def-weak}
We call $(\theta', \bar{\theta}, \mathbf u, w)$ a \emph{weak solution} on $[0,T]$ for system (\ref{inP-1})-(\ref{inP-5}) if
\begin{align*}
&\theta'  \in L^2(0,T;H^1_h(\Omega)) \cap  C([0,T];L^2(\Omega));
\;\;\; \theta'_t \in L^2(0,T;(H^1_h(\Omega))');   \notag\\
&\bar{\theta} \in   L^2(0,T; \dot H^1(0,2\pi));    \notag\\
&\Delta_h \mathbf u, \; \Delta_h w   \in     C ([0,T];L^2(\Omega));
\;\; \mathbf u_z,  w_z,  \Delta_h \omega,  \nabla_h^3 w   \in L^2(\Omega \times (0,T)),
\end{align*}
and the equations hold in the function spaces specified below:
\begin{align}
& \frac{\partial \psi}{\partial z} = \text{Ra} \, \theta' + \Delta_h w,     \;\; \text{in}  \;\; L^2(0,T;H^1_h(\Omega)) \cap  C([0,T];L^2(\Omega)),  \label{wD-1}   \\
& - \frac{\partial w}{\partial z} = \Delta_h \omega,     \;\; \text{in}  \;\; L^2(\Omega \times (0,T)),   \label{wD-2}    \\
&\frac{\partial \theta'}{\partial t}  + \mathbf u \cdot \nabla_h \theta' +  w \frac{\partial \bar{\theta}}{\partial z} = \Delta_h \theta',   \;\; \text{in}  \;\; L^2(0,T;(H^1_h(\Omega))'),    \label{wD-3} \\
&\frac{\partial(\bar{\theta' w})}{\partial z} = \frac{\partial^2 \bar \theta}{\partial z^2},      \;\; \text{in}  \;\; L^2(0,T;H^{-1}(0,2\pi)),    \label{wD-4}
\end{align}
with $\nabla_h \cdot \mathbf u=0$, $\omega =   \nabla_h \times \mathbf u$, $\psi =  \Delta_h^{-1} \omega$, and
$\bar{\theta'}=\bar{w}=\bar{\omega}=\bar{\psi}=0$, $\bar{\mathbf u}=0$, such that the initial condition $\theta'(0)=\theta'_0 \in L^2(\Omega)$ is satisfied.
\end{definition}

According to the derivation of model (\ref{m-1})-(\ref{m-5}) in \cite{SJKW}, the quantities $\theta'$, $\mathbf u$ and $w$ are ``fluctuating" quantities about the horizontal mean, i.e., the original quantities subtracted by their horizontal means. Therefore, in the above definition of weak solutions, all quantities are demanded to have horizontal average zero.

In Definition \ref{def-weak}, the horizontal mean temperature
$\bar{\theta}(z,t) = \frac{1}{4\pi^2 L^2}\int_{[0,2\pi L]^2} \theta(x,y,z,t) dx dy$ belongs to the space $\dot H^1(\Omega)$,
which demands that $\int_0^{2\pi} \bar{\theta}(z,t) dz =0$ due to (\ref{Hdot}).
Therefore, $\int_{\Omega} \theta(x,y,z,t) dx dy dz =0$, namely, the average temperature is zero for all $t \in [0,T]$.

In the next theorem we state the existence and uniqueness of global weak solutions to system (\ref{inP-1})-(\ref{inP-5}) as well as the continuous dependence on initial data.

\begin{theorem}[Global well-posedness of weak solutions]  \label{thm-weak}
Assume $\theta'_0 \in L^2(\Omega)$ with $\bar{\theta_0'}= 0$. Then, system (\ref{inP-1})-(\ref{inP-5}) has a unique weak solution
$(\theta', \bar{\theta}, \mathbf u, w)$ for all $t\geq 0$, in the sense of Definition \ref{def-weak}. Moreover, the solution satisfies the following energy equality:
\begin{align}   \label{EI}
\frac{1}{2} \|\theta'(t)\|_2^2   +  \int_0^t \|\nabla_h \theta' (s)\|_2^2 ds
+  4\pi^2 L^2     \int_0^t \int_0^{2\pi}  |\partial_z \bar{\theta}(z,s)|^2 dz ds  =  \frac{1}{2} \|\theta'_0\|_2^2,
\end{align}
for all $t\geq 0$. Also, the following decay estimates are valid:
\begin{align*}
&\|\theta'(t)\|_2^2 + \|\Delta_h \mathbf u(t)\|_2^2 +   \|\Delta_h w(t)\|_2^2  \leq C e^{-\frac{2}{L^2} t} \|\theta'_0\|_2^2,
 \;\; \text{for all} \;\; t\geq 0; \\
 &  \int_0^{2\pi} |\bar{\theta}_z(z,t)|^2 dz \leq C  e^{-\frac{4}{L^2} t} \|\theta'_0\|_2^4,     \;\; \text{for all} \;\; t\geq 0.
  \end{align*}
In addition, if $\{{\theta'_{0,n}}\}$ is a sequence of initial data in $L^2(\Omega)$ such that ${\theta'_{0,n}} \rightarrow \theta'_0$ in $L^2(\Omega)$,
then the corresponding weak solutions $\{(\theta'_n, \bar{\theta}_n, \mathbf u_n, w_n)\}$
and $(\theta',\bar{\theta},\mathbf u, w)$ with
$\theta'_n(0)=\theta'_{0,n}$ and $\theta'(0)=\theta'_0$ satisfy
$\theta'_n \rightarrow \theta'$ in $C([0,T];L^2(\Omega)) \cap L^2(0,T;H^1_h(\Omega))$,
$\bar{\theta}_n \rightarrow \bar{\theta}$ in $L^2(0,T; \dot H^1(0,2\pi))$,
and $(\mathbf u_n,w_n)\rightarrow (\mathbf u, w)$ in $L^2(0,T;H^1(\Omega))$.
\end{theorem}

\vspace{0.1 in}

\subsection{Strong solutions}
We define a strong solution of system (\ref{inP-1})-(\ref{inP-5}).
\begin{definition}   \label{def-solution}
We call $(\theta', \bar{\theta}, \mathbf u, w)$ a \emph{strong solution} on $[0,T]$ for system (\ref{inP-1})-(\ref{inP-5}) if
\begin{align}
&\theta'   \in L^{\infty} (0,T;H^1(\Omega))\cap C([0,T];L^2(\Omega)) \; \text{such that}\; \Delta_h \theta',\; \nabla_h \theta'_z \in L^2(\Omega \times (0,T))      ;       \label{reg-1} \\
&\theta'_t \in L^2(\Omega \times (0,T)) ;    \label{reg-1'} \\
&\bar{\theta} \in   L^{\infty}(0,T; \dot H^2(0,2\pi));      \label{reg-2} \\
&\Delta_h \mathbf u, \; \Delta_h w      \in     L^{\infty} (0,T;H^1(\Omega));      \;\;
 \mathbf u_z,  w_z \in L^{\infty}(0,T; L^2(\Omega)),
\label{reg-3}
\end{align}
and the equations hold in the function spaces specified below:
\begin{align}
& \frac{\partial \psi}{\partial z} = \text{Ra} \, \theta' + \Delta_h w,     \;\; \text{in}  \;\; L^{\infty}(0,T;H^1(\Omega)),     \label{inD-1} \\
& - \frac{\partial w}{\partial z} = \Delta_h \omega,     \;\; \text{in}  \;\; L^{\infty}(0,T;L^2(\Omega)),      \label{inD-2}\\
&\frac{\partial \theta'}{\partial t}  + \mathbf u \cdot \nabla_h \theta' +  w \frac{\partial \bar{\theta}}{\partial z} = \Delta_h \theta',   \;\; \text{in}  \;\; L^2(\Omega \times (0,T)),  \label{inD-3}\\
&\frac{\partial(\bar{\theta' w})}{\partial z} = \frac{\partial^2 \bar \theta}{\partial z^2},      \;\; \text{in}  \;\; L^{\infty}(0,T;L^2(0,2\pi)),   \label{inD-4}
\end{align}
with $\nabla_h \cdot \mathbf u=0$, $\omega = \nabla_h \times \mathbf u$, $\psi =  \Delta_h^{-1} \omega$, and
$\bar{\theta'}=\bar{w}=\bar{\omega}=\bar{\psi}=0$, $\bar{\mathbf u}=0$,
such that the initial condition $\theta'(0) =\theta'_0 \in H^1_0(\Omega)$ is satisfied.

\end{definition}

\vspace{0.1 in}

The following theorem states the existence and uniqueness of global strong solutions to system (\ref{inP-1})-(\ref{inP-5}).
\begin{theorem}[Global well-posedness of strong solutions]  \label{thm-strong}
Assume $\theta'_0 \in H^1(\Omega)$ with
$\bar{\theta_0'}= 0$. Then, system (\ref{inP-1})-(\ref{inP-5}) has a unique strong solution
$(\theta', \bar{\theta}, \mathbf u, w)$ for all $t\geq 0$, in the sense of Definition \ref{def-solution}. Also, energy equality (\ref{EI}) is valid.
\end{theorem}

\vspace{0.1 in}

 \section{Preliminaries}   \label{sec-ineq}
We state some inequalities which will be useful in our estimates.  Let $\Omega=[0,2\pi L]^2 \times [0,2\pi]$ be a three-dimensional fundamental periodic domain.

The following is an anisotropic Ladyzhenskaya-type inequality which has been proved in \cite{CGT}.
 \begin{lemma}   \label{lemma1}
Let $f\in H^1(\Omega)$, $\nabla_h g\in L^2(\Omega)$ and $h\in L^2(\Omega)$. Then
\begin{align}    \label{Lady}
\int_{\Omega} |fgh| dx dy dz  \leq C (\|f\|_2+\|\nabla_h f\|_2)^{\frac{1}{2}}  \left(\|f\|_2 + \|f_z\|_2 \right)^{\frac{1}{2}}
 \|g\|_2^{\frac{1}{2}}   (\|g\|_2+\|\nabla_h g\|_2)^{\frac{1}{2}} \|h\|_2.
\end{align}
\end{lemma}

 \vspace{0.1 in}

\begin{definition}
Let $s\geq 0$. We say $\partial_z^{s} f \in L^2(\Omega)$ if $f\in L^2(\Omega)$ satisfying
\begin{align*}
\|\partial_z^{s}  f \|_2^2 = \sum_{\mathbf k = (k_1,k_2,k_3) \in \mathbb Z^3} |k_3|^{2s} |\hat f (\mathbf k)|^2 < \infty.
\end{align*}
\end{definition}

\begin{lemma}    \label{lemma2}
Let $s>1/2$. Assume $f$, $\partial_z^s f \in L^2(\Omega)$, then
\begin{align}   \label{ineq-2}
\sup_{z\in [0,2\pi]} \int_{[0,2\pi L]^2} |f(x,y,z)|^2 dx dy \leq C(\|f\|_2^2+\|\partial_z^s f\|_2^2).
\end{align}
\end{lemma}
\begin{proof}
Thanks to the one-dimensional imbedding $\|\phi\|_{L^{\infty}(0,2\pi)} \leq C\|\phi\|_{H^{s}(0,2\pi)}$, when $s>1/2$, then for a.e. $z\in [0,2\pi]$,
\begin{align*}
\int_{[0,2\pi L]^2} |f(x,y,z)|^2 dx dy
\leq C\int_{[0,2\pi L]^2} \left( \int_0^{2\pi} (|f|^2 + |\partial_z^s f|^2) dz \right) dx dy = C (\|f\|_2^2 + \|\partial_z^s f\|_2^2).
\end{align*}
\end{proof}

Recall the periodic domain $\Omega=[0,2\pi L]^2 \times [0,2\pi]$. For any periodic function $f\in H^1_h(\Omega)$, i.e., $\nabla_h f \in L^2(\Omega)$ and $\bar{f}=0$, the Poincar\'e inequality is valid:
\begin{align}   \label{Poin}
\|f\|_2^2 \leq L^2 \|\nabla_h f\|_2^2.
\end{align}

Next, we state some identities which will be employed in the energy estimate. For sufficiently smooth periodic functions $\mathbf u$, $f$ and $g$ on $\Omega$, such that $\nabla_h \cdot \mathbf{u}=0$, an integration by parts shows
\begin{align}   \label{iden-0}
\int_{\Omega}(\mathbf{u} \cdot \nabla_h f) g \, dx dy dz=-  \int_{\Omega} (\mathbf{u} \cdot \nabla_h g) f \, dx dy dz.
\end{align}
This implies
\begin{align}   \label{iden-1}
\int_{\Omega} (\mathbf{u} \cdot \nabla_h f) f \, dx dy dz=0,
\end{align}
if $\nabla_h \cdot \mathbf{u}=0$.

Note that the horizontal velocity $\mathbf u$, the vertical component $\omega$ of the vorticity, and the horizontal stream function $\psi$ such that $\bar{\psi}=0$ have the following relations:
\begin{align}   \label{iden-4}
\omega=\nabla_h \times \mathbf u=v_x-u_y, \;\;\; \omega=\Delta_h \psi,\;\;\;
\mathbf u = (u,v) =(-\psi_y,\psi_x).
\end{align}

\vspace{0.1 in}

\section{Weak solutions}   \label{sec-weak}
In this section, we prove the global well-posedness of weak solutions to system (\ref{inP-1})-(\ref{inP-5})
by using the Galerkin method.

\subsection{Existence of weak solutions}
\subsubsection{Galerkin approximation}
Let $P_m$ be an orthogonal projection onto lower Fourier modes, namely,
$P_m \phi = \sum_{\substack{     \mathbf{k} \in \mathbb Z^3 \\
|\mathbf k| \leq m    }       }  \hat \phi(\mathbf k) e_{\mathbf k}$. Here,
$e_{\mathbf k} =  \frac{1}{(2\pi)^3 L^2} \exp[{i\left( \frac{k_1x+k_2 y}{L} + k_3 z \right)}]$, $\mathbf{k} \in \mathbb Z^3$, form an orthonormal basis for $L^2(\Omega)$, where $\Omega=[0,2\pi L]^2 \times [0,2\pi]$ is a three-dimensional periodic domain.

We consider the Galerkin approximation for system (\ref{inP-1})-(\ref{inP-5}):
\begin{align}
& \frac{\partial \psi_m}{\partial z} = \text{Ra}   \, \theta'_m + \Delta_h w_m,   \label{inG-1} \\
& - \frac{\partial w_m}{\partial z} = \Delta_h \omega_m,   \label{inG-2}\\
&\frac{\partial \theta'_m}{\partial t}  + P_m(\mathbf u_m \cdot \nabla_h \theta'_m) +  P_m(w_m \frac{\partial \bar{\theta_m}}{\partial z}) = \Delta_h \theta'_m,   \label{inG-3}\\
&\frac{\partial(\bar{\theta'_m w_m})}{\partial z} = \frac{\partial^2 \bar{\theta_m}}{\partial z^2},     \label{inG-4}
\end{align}
such that $\nabla_h \cdot \mathbf u_m=0$, with the initial condition $\theta'_m(0)=P_m \theta'_0$ where $\bar{\theta_0'}=0$.
Also, $\mathbf u_m=(u_m,v_m)=(-\partial_y \psi_m, \partial_x \psi_m)$, and $\omega_m=\Delta_h \psi_m$
such that $\bar{\psi_m}=0$. Moreover, the temperature $\theta_m = \theta_m' +  \bar{\theta_m}$, where the horizontal mean temperature
$\bar{\theta_m} (z,t)= \frac{1}{4\pi^2 L^2} \int_{[0,2\pi L]^2}  \theta_m(x,y,z,t)  dx dy$.
We assume the horizontal average zero condition:
$\bar{\theta_m'}=\bar{w_m}=\bar{\omega_m}=\bar{\psi_m}=0$ and $\bar{\mathbf u_m}=0$.
In addition, we demand the average value of $\theta_m$ is zero:
\begin{align} \label{Gaz}
\int_{\Omega} \theta_m(x,y,z,t) dx dy dz =0, \;\;\text{i.e.}, \;\;\int_0^{2\pi} \bar{\theta_m}(z,t) dz =0,  \;\;  \text{for all}\,\, t\geq 0.
\end{align}
Thus, the unknown functions in the Galerkin system can be expressed as finite sums of Fourier modes:
\begin{align}
&\theta'_m=
\sum_{\substack{     \mathbf{k} = (k_1, k_2, k_3)  \in \mathbb Z^3 \\ k_1^2+k_2^2 \not = 0 \\
|\mathbf k| \leq m    }       }   {\widehat{\theta'_m}(\mathbf{k},t)} e_{\mathbf k},  \;\;\;\;
\theta_m=
\sum_{\substack{     \mathbf{k} = (k_1, k_2, k_3)  \in \mathbb Z^3 \\
0<|\mathbf k| \leq m    }       }   {\widehat{\theta_m}(\mathbf{k},t)} e_{\mathbf k},  \;\;\;\;
  \label{mode-1}   \\
&  \mathbf u_m =
\sum_{\substack{     \mathbf{k} = (k_1, k_2, k_3)\in \mathbb Z^3 \\ k_1^2+k_2^2 \not = 0 \\
|\mathbf k| \leq m    }       }   {\widehat{\mathbf u_m}(\mathbf{k},t)} e_{\mathbf k},   \;\;\;\;
w_m =
\sum_{\substack{     \mathbf{k} = (k_1, k_2, k_3)\in \mathbb Z^3 \\ k_1^2+k_2^2 \not = 0 \\
|\mathbf k| \leq m    }       }   {\widehat{w_m}(\mathbf{k},t)} e_{\mathbf k}.      \label{mode-2}      \\
&\psi_m =
\sum_{\substack{     \mathbf{k} = (k_1, k_2, k_3)  \in \mathbb Z^3 \\ k_1^2+k_2^2 \not = 0 \\
|\mathbf k| \leq m    }       }   {\widehat{\psi_m}(\mathbf{k},t)} e_{\mathbf k}, \;\;\;\;
\omega_m =
\sum_{\substack{     \mathbf{k} = (k_1, k_2, k_3)  \in \mathbb Z^3 \\ k_1^2+k_2^2 \not = 0 \\
|\mathbf k| \leq m    }       }   {\widehat{\omega_m}(\mathbf{k},t)} e_{\mathbf k}.    \label{mode-3}
\end{align}

Since equations (\ref{inG-1})-(\ref{inG-2}) are linear, they can be solved explicitly if $\theta_m'$ is given. Indeed, since $\omega_m=\Delta_h \psi_m$, we have
\begin{align*}
&ik_3 {\widehat{\psi_m}(\mathbf{k})} = \text{Ra} \, {\widehat{\theta'_m}(\mathbf{k})} - \left(\frac{k_1^2+k_2^2}{L^2}\right) {\widehat{w_m}(\mathbf{k})}  \\
&-ik_3 {\widehat{w_m}(\mathbf{k})} =  \left(\frac{k_1^2+k_2^2}{L^2}\right)^2 {\widehat{\psi_m}(\mathbf{k})}.
\end{align*}
The above linear system can be written as
\begin{align}     \label{matrix}
\begin{pmatrix}
ik_3 &  \frac{k_1^2+k_2^2}{L^2} \\
-(\frac{k_1^2+k_2^2}{L^2})^2   &    -i k_3
\end{pmatrix}
\begin{pmatrix}
{\widehat{\psi_m}(\mathbf{k})}  \\
{\widehat{w_m}(\mathbf{k})}
\end{pmatrix}
=
\text{Ra}
\begin{pmatrix}
{\widehat{\theta'_m}(\mathbf{k})} \\
0
\end{pmatrix}.
\end{align}
For the $2\times 2$ matrix in (\ref{matrix}), its determinant $k_3^2 + (\frac{k_1^2+k_2^2}{L^2})^3>0$ because $k_1^2+k_2^2 \not=0$. Thus we can take the inverse of this matrix. It follows that
\begin{align*}
&\begin{pmatrix}
{\widehat{\psi_m}(\mathbf{k})}  \\
{\widehat{w_m}(\mathbf{k})}
\end{pmatrix}
=  \text{Ra}  \begin{pmatrix}
ik_3 &  \frac{k_1^2+k_2^2}{L^2} \\
-(\frac{k_1^2+k_2^2}{L^2})^2   &    -i k_3
\end{pmatrix}^{-1}
\begin{pmatrix}
{\widehat{\theta'_m}(\mathbf{k})} \\
0
\end{pmatrix} \notag\\
&= \frac{ \text{Ra} }{k_3^2 +   (\frac{k_1^2+k_2^2}{L^2})^3 }
\begin{pmatrix}
-ik_3 & - \frac{k_1^2+k_2^2}{L^2}\\
(\frac{k_1^2+k_2^2}{L^2})^2   &    i k_3
\end{pmatrix}
\begin{pmatrix}
{\widehat{\theta'_m}(\mathbf{k})} \\
0
\end{pmatrix}  =\frac{   \text{Ra}  }{k_3^2 + (\frac{k_1^2+k_2^2}{L^2})^3}
\begin{pmatrix}
-ik_3 \\
(\frac{k_1^2+k_2^2}{L^2})^2
\end{pmatrix}
 {\widehat{\theta'_m}(\mathbf{k})}.
\end{align*}
Consequently,
\begin{align}
{\widehat{\psi_m}(\mathbf{k},t)}= \text{Ra}   \left(\frac{-i  k_3}{k_3^2 + (\frac{k_1^2+k_2^2}{L^2})^3}\right)     {\widehat{\theta'_m}(\mathbf{k},t)},     \label{FourG-1}   \\
{\widehat{w_m}(\mathbf{k},t)} =    \text{Ra}  \left(\frac{(\frac{k_1^2+k_2^2}{L^2})^2}{k_3^2 + (\frac{k_1^2+k_2^2}{L^2})^3} \right)  {\widehat{\theta'_m}(\mathbf{k},t)}.      \label{FourG-2}
\end{align}
Furthermore, since $\mathbf u_m=(u_m,v_m)=(-\partial_y \psi_m, \partial_x \psi_m)$, we obtain from (\ref{FourG-1}) that
\begin{align}
&{\widehat{u_m}(\mathbf{k},t)} =   \text{Ra}   \left(\frac{-\frac{k_2}{L} k_3}{k_3^2 + (\frac{k_1^2+k_2^2}{L^2})^3}\right)     {\widehat{\theta'_m}(\mathbf{k},t)}, \;\;\;\;
{\widehat{v_m}(\mathbf{k},t)} =   \text{Ra}   \left(\frac{\frac{k_1}{L} k_3}{k_3^2 + (\frac{k_1^2+k_2^2}{L^2})^3}\right)     {\widehat{\theta'_m}(\mathbf{k},t)}.     \label{FourG-3}
\end{align}

Due to (\ref{inG-4}), we have $\frac{\partial \bar{\theta_m}}{\partial z} =  \bar{\theta'_m w_m} + c(t)$,
for some constant $c(t)$ depending only on $t$. Then, since $\bar{\theta_m}$ is periodic on $[0,2\pi]$, we obtain
  $0=\bar{\theta_m}(2\pi) -  \bar{\theta_m}(0) =   \int_0^{2\pi} \bar{\theta'_m w_m} dz+  c(t) 2\pi$,
  which implies that $c(t) =   - \frac{1}{2\pi}  \int_0^{2\pi} \bar{\theta'_m w_m} dz $.
 Therefore,
\begin{align}  \label{inG-4'}
 \frac{\partial \bar{\theta_m}}{\partial z} =  \bar{\theta'_m w_m} - \frac{1}{2\pi}  \int_0^{2\pi} \bar{\theta'_m w_m} dz.
  \end{align}

By substituting (\ref{FourG-2})-(\ref{FourG-3}) and (\ref{inG-4'}) into equation (\ref{inG-3}), we obtain a system of first order nonlinear ordinary differential equations with unknowns $\{\widehat{\theta'_m}(\mathbf k,t): \mathbf{k}\in \mathbb Z^3, k_1^2+k_2^2\not=0, |\mathbf k| \leq m\}.$ By the classical theory of ordinary differential equations, for each $m\in \mathbb N$, there exists a solution $\{\widehat{\theta'_m}(\mathbf k,t): |\mathbf k|\leq m\}$ defined on $[0,T_m^{\text{max}})$ for the system of ODEs.
Then, thanks to (\ref{FourG-1})-(\ref{FourG-3}), we obtain $\widehat{\psi_m}(\mathbf{k},t)$, $\widehat{u_m}(\mathbf{k},t)$, $\widehat{v_m}(\mathbf{k},t)$ and $\widehat{w_m}(\mathbf{k},t)$, for $|\mathbf k| \leq m$.
Next, we substitute $\theta_m'$ and $w_m$ into (\ref{inG-4'}) to get $\frac{\partial \bar {\theta_m}}{\partial z}$, and along with the assumption $\int_0^{2\pi} \bar{\theta_m}(z,t) dz =0$ from (\ref{Gaz}), we obtain $\bar{\theta_m}$. Finally, $\theta_m = \theta_m' + \bar{\theta_m}$ which satisfies $\int_{\Omega} \theta_m dx dy dz =0$.

Assume the ODE system has finite time of existence, i.e., $T_m^{\text{max}}<\infty$.
By estimate (\ref{wk-0}) below, we know that
$\|\theta'_m(t)\|_2^2   \leq   \|\theta'_m(0)\|_2^2 \leq \|\theta'_0\|_2^2$ for all $t\in [0,T_m^{\text{max}})$.
Therefore, $\theta'_m(t)$ can be extended beyond the finite time $T_m^{\text{max}}$, which is a contradiction.
It follows that $T_m^{\text{max}} = \infty$. As a result, for every $m\in \mathbb N$, the Galerkin system (\ref{inG-1})-(\ref{inG-4}) has a global solution on $[0,\infty)$.

\subsubsection{Energy estimate}   \label{sec-energy}
Taking the inner product of (\ref{inG-3}) with $\theta'_m$ and using (\ref{iden-1}), one has
\begin{align}    \label{wk-theta-1}
\frac{1}{2} \frac{d}{dt} \|\theta'_m\|_2^2 +  \|\nabla_h \theta'_m\|_2^2 +
\int_{\Omega} [P_m(w_m \frac{\partial \bar{\theta_m}}{\partial z})] \theta'_m dx dy dz =0,  \;\;  \text{for all} \; t\geq 0.
\end{align}

Recall that the horizontal mean of a function $f$ is defined as $\bar f(z)=  \frac{1}{4 \pi^2 L^2}\int_{[0, 2\pi L]^2} f(x,y,z) dx dy$.
Using $\frac{\partial(\bar{\theta'_m w_m})}{\partial z} = \frac{\partial^2 \bar{\theta_m}}{\partial z^2}$ from equation (\ref{inG-4}), we find that
\begin{align*}
&\int_{\Omega} [P_m(w_m \frac{\partial \bar{\theta_m}}{\partial z})] \theta'_m dx dy dz=\int_{\Omega} w_m \frac{\partial \bar{\theta_m}}{\partial z} \theta'_m dx dy dz
=  4\pi^2 L^2  \int_0^{2\pi} (\bar{w_m \theta'_m})   \frac{\partial \bar{\theta_m}}{\partial z}  dz    \notag\\
&= -   4\pi^2 L^2  \int_0^{2\pi}   \frac{\partial (\bar{w_m \theta'_m})}{\partial z}  \bar{\theta_m} dz
=-  4\pi^2 L^2   \int_0^{2\pi} \left(\frac{\partial^2 \bar{\theta_m}}{\partial z^2}\right)  \bar{\theta_m} dz
=   4\pi^2 L^2  \int_0^{2\pi}  \left| \frac{\partial \bar{\theta_m}}{\partial z}\right|^2 dz,
\end{align*}
for all $t\geq 0$. Therefore,
\begin{align} \label{wk-in-bL2}
\frac{1}{2} \frac{d}{dt} \|\theta'_m\|_2^2 +  \|\nabla_h \theta'_m\|_2^2
+4\pi^2 L^2  \int_0^{2\pi}  \left| \frac{\partial \bar{\theta_m}}{\partial z}\right|^2 dz =0,  \;\;  \text{for all} \; t\geq 0.
\end{align}
Integrating over $[0,t]$ yields
\begin{align}  \label{wk-0}
&\frac{1}{2}  \|\theta'_m(t)\|_2^2   + \int_0^t     \|\nabla_h \theta'_m (s)\|_2^2 ds +
4\pi^2 L^2   \int_0^t    \int_0^{2\pi}  \left|\frac{\partial \bar{\theta_m}(z,s)}{\partial z}\right|^2 dz ds   \notag\\
&= \frac{1}{2}   \|\theta'_m(0)\|_2^2  \leq   \frac{1}{2}   \|\theta'_0\|_2^2,  \;\;  \text{for all} \; t\geq 0.
\end{align}

Next, we estimate $\mathbf u_m$ and $w_m$. By (\ref{FourG-3}), we calculate
\begin{align}  \label{wk-1}
\|\Delta_h \mathbf u_m\|_2^2
&=  \text{Ra}^2    \sum_{\substack{\mathbf{k}\in \mathbb Z^3\\
                  k_1^2 + k_2^2 \not =0   }    }
\frac{     \left|\frac{k_1^2+k_2^2}{L^2}\right|^3  k_3^2     }{  \left|k_3^2 + (\frac{k_1^2+k_2^2}{L^2})^3\right|^2  }    |{\widehat{\theta'_m}(\mathbf{k})}|^2
\leq  \text{Ra}^2  \|\theta'_m\|_2^2.
\end{align}
Furthermore,
\begin{align}   \label{wk-2}
\|\partial_z \mathbf u_m\|_2^2 =    \text{Ra}^2     \sum_{\substack{\mathbf{k}\in \mathbb Z^3\\
                  k_1^2 + k_2^2 \not =0   }    }
\left|\frac{k_1^2+k_2^2}{L^2}\right|  \left|\frac{ k_3^2} {k_3^2 + (\frac{k_1^2+k_2^2}{L^2})^3}\right|^2     |{\widehat{\theta'_m}(\mathbf{k})}|^2
\leq \text{Ra}^2 \|\nabla_h \theta'_m\|_2^2.
\end{align}
In addition, by using Young's inequality, one has
\begin{align}  \label{wk-2'}
\|\partial_z^{\frac{2}{3}} \mathbf u_m\|_2^2 =   \text{Ra}^2  \sum_{\substack{\mathbf{k}\in \mathbb Z^3\\
                  k_1^2 + k_2^2 \not =0   }    }
\left|\frac{    (\frac{k_1^2 + k_2^2}{L^2})^{1/2}    k_3^{5/3}}{k_3^2 +  (\frac{k_1^2 + k_2^2}{L^2})^3} \right|^2
|\widehat{\theta'_m}(\mathbf k)|^2
\leq  \text{Ra}^2    \|\theta'_m\|_2^2.
\end{align}

By (\ref{FourG-2}), we see that
\begin{align}      \label{wk-3}
 \|\Delta_h w_m\|_2^2
=     \text{Ra}^2    \sum_{\substack{\mathbf{k}\in \mathbb Z^3\\
                  k_1^2 + k_2^2 \not =0   }    }
     \left|\frac{(\frac{k_1^2+k_2^2}{L^2})^3}{k_3^2 + (\frac{k_1^2+k_2^2}{L^2})^3} \right|^2 |\widehat {\theta'_m}(\mathbf{k})|^2
\leq     \text{Ra}^2  \|\theta'_m\|_2^2.
\end{align}
Also, using (\ref{FourG-2}), one has
\begin{align}   \label{wk-4}
\|\partial_z w_m\|_2^2 &=  \text{Ra}^2    \sum_{\substack{\mathbf{k}\in \mathbb Z^3\\
                  k_1^2 + k_2^2 \not =0   }    }
\left|\frac{(\frac{k_1^2+k_2^2}{L^2})^2 k_3}{k_3^2 + (\frac{k_1^2+k_2^2}{L^2})^3} \right|^2
  |{\widehat{\theta'_m}(\mathbf{k})}|^2  \notag\\
&= \text{Ra}^2     \sum_{\substack{\mathbf{k}\in \mathbb Z^3\\
                  k_1^2 + k_2^2 \not =0   }    }
                  \left|\frac{(\frac{k_1^2+k_2^2}{L^2})^{\frac{3}{2}} k_3}{k_3^2 + (\frac{k_1^2+k_2^2}{L^2})^3} \right|^2  \left|\frac{k_1^2+k_2^2}{L^2} \right|      | {\widehat{\theta'_m}(\mathbf{k})}|^2    \notag\\
&\leq  \text{Ra}^2   \sum_{\substack{\mathbf{k}\in \mathbb Z^3\\
                  k_1^2 + k_2^2 \not =0   }    }
                  \left|\frac{k_1^2+k_2^2}{L^2} \right|      | {\widehat{\theta'_m}(\mathbf{k})}|^2
=  \text{Ra}^2 \|\nabla_h \theta'_m\|_2^2.
\end{align}
Moreover, applying Young's inequality, we obtain
\begin{align}   \label{wk-4'}
\|\partial_z^{\frac{2}{3}}w_m\|_2^2 =   \text{Ra}^2
  \sum_{\substack{\mathbf{k}\in \mathbb Z^3\\
                  k_1^2 + k_2^2 \not =0   }    }
  \left|\frac{  |k_3|^{2/3}   (\frac{k_1^2+k_2^2}{L^2})^2}{k_3^2 + (\frac{k_1^2+k_2^2}{L^2})^3} \right|^2     |{\widehat{\theta'_m}(\mathbf{k})}|^2
  \leq \text{Ra}^2 \|\theta'_m\|_2^2.
  \end{align}

Let us fix an arbitrary time $T>0$.

Owing to estimates (\ref{wk-0})-(\ref{wk-2}), (\ref{wk-3})-(\ref{wk-4}) and
using the linear equations (\ref{inG-1})-(\ref{inG-2}), we conclude
\begin{align}
&\theta_m'  \;\; \text{is uniformly bounded in} \;\;   L^{\infty}(0,T;L^2(\Omega))\cap  L^2(0,T;H^1_h(\Omega));  \label{wk-5}  \\
& \partial_z \bar{\theta_m}   \;\; \text{is uniformly bounded in} \;\;  L^2((0,2\pi)  \times (0,T));   \label{wk-66} \\
& \mathbf u_m,  \; w_m    \;\; \text{are uniformly bounded in} \;\;   L^2(0,T;H^1(\Omega));  \label{wk-6}   \\
& \Delta_h \omega  \;\; \text{is uniformly bounded in} \;\;   L^2(\Omega \times (0,T));  \label{wk-6''} \\
& \Delta_h \mathbf u_m,  \; \Delta_h w_m, \; \partial_z \psi_m     \;\; \text{are uniformly bounded in} \;\;   L^{\infty}(0,T;L^2(\Omega))  \cap  L^2(0,T;H^1_h(\Omega)) . \label{wk-6'}
\end{align}

Therefore, on a subsequence, we have the following weak convergences as $m\rightarrow \infty$:
\begin{align}
&\theta_m'  \rightarrow   \theta',  \;  \Delta_h \mathbf u_m  \rightarrow \Delta_h \mathbf u,   \;  \Delta_h w_m \rightarrow \Delta_h w,    \; \partial_z \psi_m \rightarrow \psi_z\;   \text{weakly}^* \;\text{in} \;     L^{\infty}(0,T;L^2(\Omega));   \label{wk-7} \\
&  \theta_m'  \rightarrow   \theta',  \;  \Delta_h \mathbf u_m  \rightarrow \Delta_h \mathbf u,   \;  \Delta_h w_m \rightarrow \Delta_h w,    \; \partial_z \psi_m \rightarrow \psi_z\; \; \text{weakly in} \;       L^2(0,T;H^1_h(\Omega));    \label{wk-8} \\
&\mathbf u_m\rightarrow \mathbf u,  \; w_m\rightarrow w   \;\; \text{weakly in} \;\;   L^2(0,T;H^1(\Omega)); \label{wk-9}  \\
&\Delta_h \omega_m \rightarrow \Delta_h \omega  \;\; \text{weakly in} \;\; L^2(\Omega \times (0,T));   \label{wk-9'} \\
&  \partial_z \bar{\theta_m}     \rightarrow   \bar{\theta}_z        \;\; \text{weakly in} \;\;
L^2((0,2\pi) \times (0,T)). \label{wk-99}
\end{align}

Using these weak convergences and inequality (\ref{wk-0}), we obtain the energy inequality:
\begin{align}   \label{wk-00}
\frac{1}{2}  \|\theta'(t)\|_2^2   + \int_0^t     \|\nabla_h \theta'(s)\|_2^2 ds +
4\pi^2 L^2   \int_0^t    \int_0^{2\pi}  \left|\frac{\partial \bar{\theta}(z,s)}{\partial z}\right|^2 dz ds  \leq   \frac{1}{2}   \|\theta'_0\|_2^2,
\end{align}
for all $t\in [0,T]$.

Also, using these weak convergences, we can pass to the limit for the linear equations (\ref{inG-1})-(\ref{inG-2}) in the Galerkin approximation system to obtain
\begin{align}
& \frac{\partial \psi}{\partial z} = \text{Ra} \, \theta' + \Delta_h w,     \;\; \text{in}  \;\; L^2(0,T;H^1_h(\Omega)) \cap  L^{\infty}(0,T;L^2(\Omega)),    \label{plin1} \\
& - \frac{\partial w}{\partial z} = \Delta_h \omega,     \;\; \text{in}  \;\; L^2(\Omega \times (0,T)).  \label{plin2}
\end{align}
Then we can use the same calculations as (\ref{wk-1})-(\ref{wk-4'}), to derive that
\begin{align}
&\|\Delta_h \mathbf u\|_2^2  +    \|\Delta_h w\|^2_2 \leq 2\text{Ra}^2\|\theta'\|_2^2    \;\; \text{and}  \;\;
\|\mathbf u_z\|_2^2  +    \|w_z\|^2_2  \leq  2\text{Ra}^2 \|\nabla_h \theta'\|_2^2,    \label{linn-1}\\
&  \|\partial_z^{\frac{2}{3}}\mathbf u\|_2^2  +    \|\partial_z^{\frac{2}{3}}w\|^2_2  \leq  2\text{Ra}^2 \|\theta'\|_2^2,\label{linn-2}
\end{align}
for all $t\in [0,T]$.

\subsubsection{Passage to the limit}  \label{sec-pass}
In order to pass to the limit for the nonlinear equation (\ref{inG-3}) in the Galerkin approximation system, we shall derive certain strong convergence,
besides the already known weak convergences (\ref{wk-7})-(\ref{wk-99}). To this purpose, one has to find a uniform bound for the time derivative $\partial_t \theta'_m$, in a certain function space. From equation (\ref{inG-3}), we know
\begin{align}  \label{wk-10}
\frac{\partial \theta'_m}{\partial t}  = -P_m(\mathbf u_m \cdot \nabla_h \theta'_m) -  P_m(w_m \frac{\partial \bar{\theta_m}}{\partial z}) + \Delta_h \theta'_m.
\end{align}
We aim to find a uniform bound for each term on the right-hand side of (\ref{wk-10}).

For any $\varphi \in L^2(0,T;H^1_h(\Omega))$, applying identity (\ref{iden-0}) and Lemma \ref{lemma1}, we estimate
\begin{align}    \label{wk-12}
&\int_0^T \int_{\Omega} [P_m(\mathbf u_m \cdot \nabla_h \theta'_m)] \varphi dx dy dz dt
= -\int_0^T \int_{\Omega} (\mathbf u_m \cdot \nabla_h P_m \varphi) \theta'_m dx dy dz dt \notag\\
&\leq C\int_0^T  \|\nabla_h \mathbf u_m\|^{1/2} \left( \|\mathbf u_m\|_2 + \|\partial_z \mathbf u_m\|_2 \right)^{1/2} \|\theta'_m\|_2^{1/2}
\|\nabla_h \theta'_m\|_2^{1/2} \|\nabla_h P_m \varphi\|_2  dt \notag\\
&\leq    C\int_0^T  \|\theta'_m\|_2
\|\nabla_h \theta'_m\|_2 \|\nabla_h \varphi\|_2  dt   \notag\\
&\leq C \|\theta_0'\|_2    \left(\int_0^T   \|\nabla_h \theta'_m(t)\|_2^2 dt\right)^{1/2}
\left(\int_0^T \|\nabla_h \varphi(t)\|_2^2  dt\right)^{1/2} \notag\\
&\leq C \|\theta_0'\|_2^2  \|\varphi\|_{L^2(0,T;H^1_h(\Omega))},
\end{align}
where we have used estimates (\ref{wk-0}), (\ref{wk-1}) and (\ref{wk-2}).

Hence
\begin{align}   \label{wk-13}
\|P_m (\mathbf u_m \cdot \nabla_h \theta'_m)\|_{L^2(0,T;(H^1_h(\Omega))')} \leq C \|\theta_0'\|_2^2.
\end{align}

Furthermore, for any function $\phi \in L^2(\Omega \times (0,T))$,
\begin{align}   \label{wk-14}
&\int_0^T \int_{\Omega} [P_m(w_m \frac{\partial \bar{\theta_m}}{\partial z})]  \phi dx dy dz dt
=\int_0^T  \int_0^{2\pi} \Big(\int_{[0,2\pi L]^2} w_m (P_m \phi) dx dy\Big)   \frac{\partial \bar{\theta_m}}{\partial z}    dz dt
\notag\\
&\leq \int_0^T  \int_0^{2\pi}     \Big(\int_{[0,2\pi L]^2} |w_m|^2  dx dy\Big)^{1/2}
 \Big(\int_{[0,2\pi L]^2} |P_m \phi|^2 dx dy\Big)^{1/2}
 \Big|\frac{\partial \bar{\theta_m}}{\partial z}   \Big| dz dt         \notag\\
&\leq  C\int_0^T  \Big[\sup_{z\in [0,2\pi]} \Big(\int_{[0,2\pi L]^2} |w_m|^2 dx dy \Big)^{1/2}\Big]
\|P_m \phi\|_2 \Big( \int_0^{2\pi }  |\partial_z \bar{\theta_m}|^2 dz\Big)^{1/2}    dt \notag\\
&\leq   C\int_0^T  ( \|w_m\|_2 + \|\partial_z^{\frac{2}{3}} w_m\|_2  )       \|\phi\|_2
\left( \int_0^{2\pi }  |\partial_z \bar{\theta_m}|^2 dz\right)^{1/2}   dt \notag\\
&\leq   C\int_0^T  \|\theta'_m\|_2      \|\phi\|_2
\left( \int_0^{2\pi }  |\partial_z \bar{\theta_m}|^2 dz\right)^{1/2}   dt \notag\\
&\leq   C \|\theta_0'\|_2    \left(\int_0^T\|\phi\|_2^2 dt\right)^{1/2}     \left(\int_0^T  \int_0^{2\pi }        |\partial_z \bar{\theta_m}|^2 dz  dt\right)^{1/2}  \notag\\
&\leq   C  \|\theta_0'\|_2^2  \|\phi\|_{L^2(\Omega \times (0,T))},
\end{align}
where we have used Lemma \ref{lemma2} and estimates (\ref{wk-0}), (\ref{wk-3}) and (\ref{wk-4'}).

Thus
\begin{align}   \label{wk-15}
\|P_m(w_m  \frac{\partial  \bar{\theta_m}}{\partial z})\|_{L^2(\Omega \times (0,T))} \leq C\|\theta_0\|_2^2.
\end{align}

Also, it is easy to verify that $\|\Delta_h \theta_m'\|_{L^2(0,T;(H^1_h(\Omega))')}  \leq \|\theta_0'\|_2$. Therefore, owing to (\ref{wk-10}), (\ref{wk-13}) and (\ref{wk-15}), we obtain
\begin{align}  \label{wk-16}
\|\partial_t \theta_m'\|_{L^2(0,T;(H^1_h(\Omega))')} \leq C\|\theta_0'\|_2^2 +\|\theta_0'\|_2.
\end{align}
By the uniform bound (\ref{wk-16}), it follows that, there exists a subsequence satisfying
\begin{align}    \label{wk-177}
\partial_t \theta_m'  \rightarrow   \partial_t \theta'   \;\; \text{weakly in}  \;\; L^2(0,T;(H^1_h(\Omega))').
\end{align}
Using (\ref{FourG-2}), one has
\begin{align}   \label{wk-19}
&\|\partial_t w_m\|_2^2 =  \sum_{\substack{     \mathbf{k} = (k_1, k_2, k_3)  \in \mathbb Z^3 \\ k_1^2+k_2^2 \not = 0 \\
|\mathbf k| \leq m    }       }    |\partial_t \widehat{w_m}(\mathbf k,t)|^2
= \text{Ra}^2  \sum_{\substack{     \mathbf{k} = (k_1, k_2, k_3)  \in \mathbb Z^3 \\ k_1^2+k_2^2 \not = 0 \\
|\mathbf k| \leq m    }       }    \left|\frac{(\frac{k_1^2+k_2^2}{L^2})^{2}}{k_3^2 + (\frac{k_1^2+k_2^2}{L^2})^3} \right|^2  |{\partial_t \widehat{\theta'_m}(\mathbf{k},t)}|^2   \notag\\
&\leq  \text{Ra}^2  \sum_{\substack{     \mathbf{k} = (k_1, k_2, k_3)  \in \mathbb Z^3 \\ k_1^2+k_2^2 \not = 0 \\
|\mathbf k| \leq m    }       }   \left|\frac{L^2}{k_1^2 + k_2^2} \right|^2 |{\partial_t \widehat{\theta'_m}(\mathbf{k},t)}|^2
\leq     \text{Ra}^2  L^2 \sum_{\substack{     \mathbf{k} = (k_1, k_2, k_3)  \in \mathbb Z^3 \\ k_1^2+k_2^2 \not = 0 \\
|\mathbf k| \leq m    }       }   \frac{L^2}{k_1^2 + k_2^2} |{\partial_t \widehat{\theta'_m}(\mathbf{k},t)}|^2 \notag\\
&\leq C \|\partial_t \theta'_m\|_{(H^1_h(\Omega))'}^2.
\end{align}
Due to (\ref{wk-16}) and (\ref{wk-19}), we obtain
\begin{align}     \label{wk-21}
\partial_t w_m \;\; \text{is uniformly bounded in}  \;\; L^2(\Omega \times (0,T)).
\end{align}
In a similar manner, one can show
\begin{align}    \label{wk-22}
\partial_t \mathbf u_m \;\; \text{is uniformly bounded in}  \;\; L^2(\Omega \times (0,T)).
\end{align}

By virtue of (\ref{wk-6}), (\ref{wk-21}) and (\ref{wk-22}) and using Aubin Compactness Theorem (see, e.g., \cite{Temam}), one can extract a subsequence such that
\begin{align}     \label{wk-23}
\mathbf u_m\rightarrow \mathbf u,  \; w_m\rightarrow w  \;\; \text{strongly in} \;\;   L^2(\Omega \times (0,T)).
\end{align}

Now we can pass to the limit for the nonlinear terms of equation (\ref{inG-3}) in the Galerkin system.
Indeed, let
$\eta=\sum_{|\mathbf k|\leq N} \hat{\eta}(\mathbf k,t) e_{\mathbf k}$ be a trigonometric polynomial with continuous coefficients,
where $e_{\mathbf k}=  \frac{1}{(2\pi)^3 L^2} \exp[{i\left( \frac{k_1x+k_2 y}{L} + k_3 z \right)}]$, $\mathbf k=(k_1,k_2,k_3) \in \mathbb Z^3$. Then,
\begin{align}    \label{wk-24}
&\lim_{m\rightarrow \infty}\int_0^T \int_{\Omega} (P_m(\mathbf u_m \cdot \nabla_h \theta'_m)) \eta dx dy dz dt
= \lim_{m\rightarrow \infty}\int_0^T \int_{\Omega} (\mathbf u_m \cdot \nabla_h \theta'_m) \eta dx dy dz dt
  \notag\\
&=\lim_{m \rightarrow \infty} \int_0^T \int_{\Omega} ((\mathbf u_m - \mathbf u) \cdot \nabla_h \theta'_m) \eta dx dy dz dt
+ \lim_{m\rightarrow \infty} \int_0^T \int_{\Omega} (\mathbf u \cdot \nabla_h \theta'_m) \eta dx dy dz dt \notag\\
&=\int_0^T \int_{\Omega} (\mathbf u \cdot \nabla_h \theta') \eta dx dy dz dt,
\end{align}
by virtue of the fact that $\mathbf u_m \rightarrow \mathbf u$ in $L^2(\Omega \times (0,T))$, and the fact that $\nabla_h \theta'_m$ is uniformly bounded in $L^2(\Omega \times (0,T))$, as well as the fact that
$\nabla_h \theta'_m \rightarrow \nabla_h \theta'$ weakly in $L^2(\Omega \times (0,T))$.

For the other nonlinear term in (\ref{inG-3}), we have
\begin{align}    \label{wk-25}
&\lim_{m\rightarrow \infty}  \int_0^T \int_{\Omega} P_m\Big(w_m \frac{\partial{\bar{\theta_m}}}{\partial z}  \Big) \eta dx dy dz dt
=  \lim_{m\rightarrow \infty}  \int_0^T \int_{\Omega} \Big(w_m \frac{\partial{\bar{\theta_m}}}{\partial z}  \Big) \eta dx dy dz dt
\notag\\
&=   \lim_{m\rightarrow \infty}    \int_0^T \int_{\Omega} (w_m -w)   \frac{\partial{\bar{\theta_m}}}{\partial z} \eta dx dy dz dt
+\lim_{m\rightarrow \infty} \int_0^T \int_{\Omega} w \frac{\partial{\bar{\theta_m}}}{\partial z} \eta dx dy dz dt \notag\\
&=   \int_0^T \int_{\Omega} w \frac{\partial{\bar{\theta}}}{\partial z} \eta dx dy dz dt,
\end{align}
due to the fact that $w_m \rightarrow w$ strongly in $L^2(\Omega \times (0,T))$, and the fact that $\frac{\partial{\bar{\theta_m}}}{\partial z} $ is uniformly bounded in  $L^2((0,2\pi) \times (0,T))$, along with the fact that
$\frac{\partial{\bar{\theta_m}}}{\partial z}  \rightarrow   \frac{\partial{\bar{\theta}}}{\partial z} $ weakly in $L^2((0,2\pi) \times (0,T))$.

Therefore, we have
\begin{align}  \label{wk-255}
\int_0^T \int_{\Omega} (\frac{\partial \theta'}{\partial t}  + \mathbf u \cdot \nabla_h \theta' +  w \frac{\partial \bar{\theta}}{\partial z} - \Delta_h \theta') \eta dx dy dz dt =0,
\end{align}
for any trigonometric polynomial $\eta$ with continuous coefficients.

Using similar estimates as (\ref{wk-12}) and (\ref{wk-14}), we can derive
$\mathbf u \cdot \nabla_h \theta' \in  L^2(0,T;(H^1_h(\Omega))')$
and $ w \frac{\partial \bar{\theta}}{\partial z} \in L^2(\Omega \times (0,T))$, thus $\frac{\partial \theta'}{\partial t}  + \mathbf u \cdot \nabla_h \theta' +  w \frac{\partial \bar{\theta}}{\partial z} - \Delta_h \theta'\in L^2(0,T;(H^1_h(\Omega))')$. Then, we conclude from (\ref{wk-255}) that
\begin{align}    \label{heat}
\frac{\partial \theta'}{\partial t}  + \mathbf u \cdot \nabla_h \theta' +  w \frac{\partial \bar{\theta}}{\partial z} - \Delta_h \theta' = 0,    \;\; \text{in} \;   L^2(0,T;(H^1_h(\Omega))').
\end{align}

Next we aim to pass to the limit for the nonlinear term $\frac{\partial(\bar{\theta_m'  w_m})}{\partial z}$ in equation (\ref{inG-4}). To this end, we shall first show
$\frac{\partial(\bar{\theta'  w})}{\partial z}  \in L^2(0,T; H^{-1}(0,2\pi))$. Indeed, for any $\varphi\in L^2(0,T;\dot H^1(0,2\pi))$, we use Lemma \ref{lemma2} as well as estimates (\ref{wk-00}) and (\ref{linn-1}), to derive
\begin{align*}
&\int_0^{T} \int_0^{2\pi}  (\bar{\theta'  w})  \varphi_z dz dt
=  \frac{1}{4\pi^2 L^2}  \int_0^T   \int_0^{2\pi}  \Big(\int_{[0,2\pi L]^2} \theta' w dx dy\Big) \varphi_z dz dt  \notag\\
&\leq C   \int_0^T \int_0^{2\pi} \Big(\int_{[0, 2\pi L]^2} |\theta'|^2 dx dy \Big)^{1/2}
\Big(\int_{[0, 2\pi L]^2} |w|^2 dx dy \Big)^{1/2}   |\varphi_z| dz dt  \notag\\
&\leq  C\int_0^T  \sup_{z\in [0,2\pi]} \Big(\int_{[0,2\pi L]^2} |w|^2 dx dy\Big)^{1/2}  \|\theta'\|_2
\Big(\int_0^{2\pi} |\varphi_z|^2 dz\Big)^{1/2} dt  \notag\\
&\leq C \|\theta'_0\|_2  \int_0^T    (\|w\|_2 + \|w_z\|_2)   \Big(\int_0^{2\pi} |\varphi_z|^2 dz\Big)^{1/2} dt     \notag\\
&\leq C  \|\theta'_0\|_2  \int_0^T  \|\nabla_h \theta'\|_2   \Big(\int_0^{2\pi} |\varphi_z|^2 dz\Big)^{1/2} dt    \notag\\
&\leq C  \|\theta'_0\|_2  \Big( \int_0^T \|\nabla_h \theta'\|_2^2 dt   \Big)^{1/2}  \Big( \int_0^T \int_0^{2\pi} |\varphi_z|^2 dz dt   \Big)^{1/2}      \notag\\
&\leq C   \|\theta'_0\|_2^2 \|\varphi\|_{L^2(0,T;\dot H^1(0,2\pi))}.
\end{align*}
It follows that $\frac{\partial(\bar{\theta'  w})}{\partial z}  \in L^2(0,T; H^{-1}(0,2\pi))$ and $\|\frac{\partial(\bar{\theta'  w})}{\partial z}\|_{L^2(0,T; H^{-1}(0,2\pi))} \leq  C   \|\theta'_0\|_2^2$.

Now we take a test function $\xi=\sum_{0<|j| \leq N}  \hat{\xi}(j,t) e^{ij z}$ where Fourier coefficients $\hat{\xi}(t,j)$ are continuous in $t$. By using the fact that $w_m \rightarrow w$ strongly in $L^2(\Omega \times (0,T))$ and that $\theta_m' \rightarrow \theta'$ weakly in $L^2(\Omega \times (0,T))$, we derive
\begin{align}   \label{wk-26}
&\lim_{m\rightarrow \infty}\int_0^T \int_0^{2\pi} \frac{\partial (\bar{\theta_m'  w_m})}{\partial z} \xi dz dt
=  - \lim_{m\rightarrow \infty}\int_0^T \int_0^{2\pi} (\bar{\theta_m'  w_m}) \xi_z dz dt
\notag\\
&= -\lim_{m\rightarrow \infty} \frac{1}{4\pi^2 L^2}\int_0^T \int_{\Omega} \theta_m' (w_m -w) \xi_z dx dy dz dt
- \lim_{m\rightarrow \infty} \frac{1}{4\pi^2 L^2}\int_0^T \int_{\Omega} \theta_m' w  \xi_z dx dy dz dt \notag\\
&=- \frac{1}{4\pi^2 L^2}\int_0^T \int_{\Omega} \theta' w  \xi_z dx dy dz dt  = - \int_0^T \int_0^{2\pi} (\bar{\theta' w}) \xi_z dz dt  \notag\\
&= \int_0^T  \langle \frac{\partial (\bar{\theta' w})}{\partial z}, \xi \rangle_{{H^{-1}(0,2\pi)} \times \dot H^1(0,2\pi)} dt.
\end{align}
Also, since $\partial_z \bar{\theta_m} \rightarrow \bar{\theta}_z$ weakly in $L^2((0,2\pi) \times (0,T))$, we obtain
\begin{align}   \label{wk-27}
&\lim_{m\rightarrow \infty}\int_0^T \int_0^{2\pi} \frac{\partial^2 \bar{\theta_m}}{\partial z^2} \xi dz dt
=  -    \lim_{m\rightarrow \infty}\int_0^T \int_0^{2\pi} \frac{\partial \bar{\theta_m}}{\partial z} \xi_z dz dt    \notag\\
&=  -\int_0^T \int_0^{2\pi} \frac{\partial \bar{\theta}}{\partial z} \xi_z dz dt
= \int_0^T  \langle   \frac{\partial^2 \bar{\theta}}{\partial z^2}, \xi \rangle_{{H^{-1}(0,2\pi)} \times \dot H^1(0,2\pi)} dt.\end{align}

Because $\frac{\partial(\bar{\theta'_m w_m})}{\partial z} = \frac{\partial^2 \bar{\theta_m}}{\partial z^2}$ and due to (\ref{wk-26})-(\ref{wk-27}), we obtain
\begin{align} \label{wk-29}
 \int_0^T  \langle \frac{\partial (\bar{\theta' w})}{\partial z}  -    \frac{\partial^2 \bar{\theta}}{\partial z^2}    , \xi \rangle_{{H^{-1}(0,2\pi)} \times \dot H^1(0,2\pi)} dt =0.
 \end{align}
Then, since $\frac{\partial (\bar{\theta' w})}{\partial z}$ and $\frac{\partial^2 \bar{\theta}}{\partial z^2}$
both belong to $L^2(0,T; H^{-1}(0,2\pi))$, we conclude that
\begin{align}   \label{wk-30}
\frac{\partial (\bar{\theta' w})}{\partial z}  =  \frac{\partial^2 \bar{\theta}}{\partial z^2},   \;\;  \text{in}    \;  L^2(0,T; H^{-1}(0,2\pi))
\end{align}

In sum, because of (\ref{plin1})-(\ref{plin2}), (\ref{heat}) and (\ref{wk-30}), we have obtained a weak solution
for system (\ref{inP-1})-(\ref{inP-5}) on $[0,T]$, in the sense of Definition \ref{def-weak}.
Then, thanks to the energy inequality (\ref{wk-00}), the solution can be extended to a global weak solution on $[0,\infty)$.
This completes the proof for the global existence of weak solutions for system (\ref{inP-1})-(\ref{inP-5}).

\subsubsection{Energy identity}
Let $T>0$. Since $$\frac{\partial \theta'}{\partial t} \in L^2(0,T;(H^1_h(\Omega))')  \;\; \text{and} \;\; \theta' \in L^2(0,T;H^1_h(\Omega)),$$
then according to Lemma 2.1 on page 176 of Temam's book \cite{Temam}, we obtain $\theta'\in C([0,T];L^2(\Omega))$ and
\begin{align}  \label{dual}
\frac{d}{dt}\|\theta'\|_2^2 =   2 \langle  \frac{\partial \theta'}{\partial t}, \theta'  \rangle_{(H^1_h(\Omega))' \times H^1_h(\Omega)}.
\end{align}
Therefore, we can take scalar product of (\ref{heat}) with $\theta'$ to obtain
\begin{align}   \label{wk-28}
\frac{1}{2} \frac{d}{dt} \|\theta'(t)\|_2^2   +  \|\nabla_h \theta' (t)\|_2^2
+  4\pi^2 L^2     \int_0^{2\pi}  |\partial_z \bar{\theta}(z,t)|^2 dz  =  0,  \;\;  \text{for all} \; t\in [0,T].
\end{align}
Integrating (\ref{wk-28}) over $[0,t]$ yields the energy identity (\ref{EI}).

\subsubsection{Decay of the solution} Since $\bar{\theta'}=0$, then the Poincar\'e inequality (\ref{Poin}) shows that
$  \frac{1}{L^2} \|\theta' (t)\|_2^2 \leq \|\nabla_h \theta' (t)\|_2^2$. Therefore, it follows from (\ref{wk-28}) that
\begin{align*}
\frac{1}{2} \frac{d}{dt} \|\theta'(t)\|_2^2   +  \frac{1}{L^2} \|\theta' (t)\|_2^2  \leq 0,  \;\; \text{for all} \; t\geq 0.
\end{align*}
Consequently,
\begin{align}  \label{deca-1}
\|\theta'(t)\|_2^2 \leq e^{-\frac{2}{L^2} t} \|\theta'_0\|_2^2,     \;\; \text{for all} \; t\geq 0.
\end{align}
Then, due to (\ref{linn-1}) and (\ref{deca-1}), we obtain
\begin{align}   \label{deca-2}
\|\Delta_h \mathbf u(t)\|_2^2 + \|\Delta_h w(t)\|_2^2 \leq C e^{-\frac{2}{L^2} t} \|\theta'_0\|_2^2,   \;\;  \text{for all} \; t\geq 0.
\end{align}
Due to the Poincar\'e inequality (\ref{Poin}), we have also $\|\mathbf u(t)\|_2^2 + \|w(t)\|_2^2 \leq C e^{-\frac{2}{L^2} t} \|\theta'_0\|_2^2$.

Next, taking the scalar product of equation (\ref{wk-30}) and $\bar{\theta}$ in the duality between $H^{-1}(0,2\pi)$
and $\dot H^1(0,2\pi)$ yields
\begin{align}  \label{deca-3}
\int_0^{2\pi} |\bar{\theta}_z|^2 dz  = \int_0^{2\pi} (\bar{\theta' w})\bar{\theta}_z dz
\leq  \frac{1}{2}  \int_0^{2\pi} |\bar{\theta}_z|^2 dz + \frac{1}{2} \int_0^{2\pi} \left|\bar{\theta' w} \right|^2 dz,
\end{align}
where we use the Cauchy-Schwarz inequality and the Young's inequality.

By (\ref{deca-3}), we have
\begin{align*}
&\int_0^{2\pi} |\bar{\theta}_z|^2 dz
\leq   \int_0^{2\pi} \left|\bar{\theta' w} \right|^2 dz
\leq C\int_0^{2\pi} \Big(\int_{[0,2\pi L]^2} |\theta'|^2 dx dy\Big)   \Big(\int_{[0,2\pi L]^2} |w|^2 dx dy\Big)  dz   \notag\\
&\leq C \sup_{z\in [0,2\pi]} \Big(\int_{[0,2\pi L]^2} |w|^2 dx dy\Big)    \|\theta'\|_2^2
\leq C  (\|w\|_2^2 + \|\partial_z^{\frac{2}{3}} w\|_2^2)   \|\theta'\|_2^2  \leq C \|\theta'\|_2^4
\leq C  e^{-\frac{4}{L^2} t} \|\theta'_0\|_2^4,
\end{align*}
where we have used Lemma \ref{lemma2} and (\ref{linn-2}) as well as the decay estimate (\ref{deca-1}).
Since $\int_0^{2\pi} \bar{\theta} dz =0$, then we can use the Poincar\'e inequality to conclude
\begin{align*}
\int_0^{2\pi} |\bar{\theta}|^2 dz \leq C\int_0^{2\pi} |\bar{\theta}_z|^2 dz \leq C  e^{-\frac{4}{L^2} t} \|\theta'_0\|_2^4.
\end{align*}

\vspace{0.1 in}

\subsection{Uniqueness of weak solutions and continuous dependence on initial data}
This section is devoted to proving the uniqueness of weak solutions. Assume there are two weak solutions
$(\theta'_1, \bar{\theta_1}, \mathbf u_1,w_1)$ and  $(\theta'_2, \bar{\theta_2}, \mathbf u_2,w_2)$ on $[0,T]$,
in the sense of Definition \ref{def-weak}. Set
 $\theta'=\theta_1'-\theta_2'$, $\bar{\theta}=\bar{\theta_1}  -  \bar{\theta_2}$, $\theta=\theta_1 - \theta_2$, $\mathbf u=\mathbf u_1-\mathbf u_2$, $w=w_1-w_2$, $\psi=\psi_1 - \psi_2$ and $\omega = \omega_1  -  \omega_2$. Therefore,
\begin{align}
& \frac{\partial \psi}{\partial z} = \text{Ra}\,\theta' + \Delta_h w, \;\; \text{in}
\;\; L^2(0,T; H^1_h(\Omega)) \cap C([0,T];L^2(\Omega)), \label{inu-1} \\
& - \frac{\partial w}{\partial z} = \Delta_h \omega,          \;\; \text{in}  \;\; L^2(\Omega \times (0,T)),        \label{inu-2}\\
&\frac{\partial \theta'}{\partial t}  + \mathbf u \cdot \nabla_h \theta'_1 +  \mathbf u_2 \cdot \nabla_h \theta' + w\frac{\partial \bar{\theta_1}}{\partial z} +
w_2 \frac{\partial \bar{\theta}}{\partial z}= \Delta_h \theta',       \;\; \text{in}  \;\; L^2(0,T;(H^1_h(\Omega))'),           \label{inu-3}\\
&\frac{\partial(\bar{\theta' w_2})}{\partial z}  +     \frac{\partial(\bar{\theta'_1 w})}{\partial z} = \frac{\partial^2 \bar \theta}{\partial z^2}, \;\; \text{in}  \;\; L^2(0,T;H^{-1}(0,2\pi)),  \label{inu-4}
\end{align}
and $\omega = \nabla_h \times \mathbf u$, $\psi = \Delta_h^{-1} \omega$.

Due to the linear equations (\ref{inu-1})-(\ref{inu-2}) and using the same calculations as (\ref{wk-1})-(\ref{wk-4'}), we have
\begin{align}
&\|\Delta_h \mathbf u\|_2^2  +    \|\Delta_h w\|^2_2 \leq 2\text{Ra}^2\|\theta'\|_2^2    \;\; \text{and}  \;\;
\|\mathbf u_z\|_2^2  +    \|w_z\|^2_2  \leq  2\text{Ra}^2 \|\nabla_h \theta'\|_2^2,    \label{lin-1}\\
&  \|\partial_z^{\frac{2}{3}}\mathbf u\|_2^2  +    \|\partial_z^{\frac{2}{3}}w\|^2_2  \leq  2\text{Ra}^2 \|\theta'\|_2^2.\label{lin-2}
\end{align}

Thanks to (\ref{dual}), we can take the scalar product of equation (\ref{inu-3}) and $\theta'$ in the duality between $(H^1_h(\Omega))'$ and $H^1_h(\Omega)$ to obtain
\begin{align}  \label{inu-5}
&\frac{1}{2}\frac{d}{dt} \|\theta'(t)\|_2^2 + \|\nabla_h \theta'(t)\|_2^2
 \notag\\
&=-\int_{\Omega} (\mathbf u \cdot  \nabla_h \theta'_1)  \theta'  dx dy dz  - \int_{\Omega}   w\frac{\partial \bar{\theta_1}}{\partial z}  \theta' dx dy dz  - \int_{\Omega}    w_2 \frac{\partial \bar{\theta}}{\partial z}  \theta' dx dy dz,
\end{align}
 where we have used the fact that $\int_{\Omega} \left(\mathbf u_2 \cdot \nabla_h \theta' \right) \theta' dx dy dz =0$ since $\nabla_h \cdot \mathbf u_2=0$.

Next we estimate each term on the right-hand side of (\ref{inu-5}).
By employing identity (\ref{iden-0}) and Lemma \ref{lemma1} as well as estimate (\ref{lin-1}), we deduce
\begin{align}      \label{inu-6}
 &-\int_{\Omega} (\mathbf u \cdot  \nabla_h \theta'_1)  \theta'  dx dy dz
 = \int_{\Omega} (\mathbf u \cdot  \nabla_h \theta')  \theta'_1  dx dy dz \notag\\
 &\leq C\|\nabla_h \mathbf u\|_2^{1/2} \left( \|\mathbf u\|_2  + \|\mathbf u_z\|_2 \right)^{1/2}     \|\nabla_h \theta'\|_2
 \|\theta_1'\|_2^{1/2} \|\nabla_h \theta_1'\|_2^{1/2} \notag\\
  &\leq C  \|\theta'\|_2^{1/2}  \|\nabla_h \theta'\|_2^{3/2}     \|\theta_1'\|_2^{1/2} \|\nabla_h \theta_1'\|_2^{1/2}  \notag\\
 &\leq  \frac{1}{2}   \|\nabla_h \theta'\|_2^2 + C \|\theta'\|_2^2     \|\theta_1'\|_2^2 \|\nabla_h \theta_1'\|_2^2,
  \end{align}
where we use H\"older's inequality.

To treat the second integral on the right-hand of (\ref{inu-5}), we apply the Cauchy-Schwarz inequality to deduce
 \begin{align}   \label{inu-7}
  &\int_{\Omega}  \big|w\frac{\partial \bar{\theta_1}}{\partial z}  \theta' \big| dx dy dz
   =  \int_0^{2\pi}   \Big(\int_{[0,2\pi L]^2} |w \theta'| dx dy\Big) | \partial_z {\bar{\theta_1}} |     dz \notag\\
  &\leq \int_0^{2\pi} \Big(\int_{[0,2\pi L]^2} |w(x,y,z)|^2 dx dy\Big)^{1/2}    \Big(\int_{[0,2\pi L]^2} |\theta'(x,y,z)|^2 dx dy\Big)^{1/2}
   |\partial_z {\bar{\theta_1}(z)}| dz \notag\\
  &\leq \Big[\sup_{z\in [0,2\pi]}  \Big(\int_{[0,2\pi L]^2} |w(x,y,z)|^2 dx dy\Big)^{1/2} \Big]  \|\theta'\|_2 \Big(\int_0^{2\pi}
  |\partial_z {\bar{\theta_1}(z)}|^2 dz\Big)^{1/2} \notag\\
  &\leq  C(\|w\|_2+ \|\partial_z^{\frac{2}{3}}w\|_2)  \|\theta'\|_2 \Big(\int_0^{2\pi}    |\partial_z {\bar{\theta_1}}|^2 dz\Big)^{1/2}  \notag\\
    &\leq   C \|\theta'\|_2^2   \Big(\int_0^{2\pi}    |\partial_z {\bar{\theta_1}}|^2 dz\Big)^{1/2},
  \end{align}
where we have used Lemma \ref{lemma2} and estimates (\ref{lin-1})-(\ref{lin-2}).

Next, we deal with the third integral on the right-hand side of (\ref{inu-5}). Indeed, from (\ref{inu-4}),
we have $\frac{\partial(\bar{\theta' w_2})}{\partial z}=\frac{\partial^2 \bar \theta}{\partial z^2}  -  \frac{\partial(\bar{\theta'_1 w})}{\partial z}$ in $L^2(0,T;H^{-1}(0,2\pi))$, and thus
 \begin{align}   \label{inu-8}
 &- \frac{1}{4 \pi^2 L^2}\int_{\Omega}    w_2 \frac{\partial \bar{\theta}}{\partial z}  \theta' dx dy dz
 =-\int_0^{2\pi}  (\bar{w_2   \theta'})    \frac{\partial \bar{\theta}}{\partial z}  dz    \notag\\
 & =   \langle \frac{\partial ({\bar{w_2     \theta'} }) }{\partial z}, \bar{\theta}  \rangle_{H^{-1}(0,2\pi) \times \dot H^1(0,2\pi)}
 =  \langle \frac{\partial^2 \bar \theta}{\partial z^2}  -  \frac{\partial(\bar{\theta'_1 w})}{\partial z},   \bar{\theta} \rangle_{H^{-1}(0,2\pi) \times \dot H^1(0,2\pi)}   \notag\\
 &= -\int_0^{2\pi} \left|\bar{\theta}_z\right|^2 dz  + \int_0^{2\pi}    (\bar{\theta'_1 w})  \bar{\theta}_z dz   \notag\\
 &\leq   -\int_0^{2\pi} \left|\bar{\theta}_z\right|^2 dz  + \frac{1}{2}\int_0^{2\pi}    \left|\bar{\theta'_1 w}\right|^2 dz  + \frac{1}{2} \int_0^{2\pi}    \left|\bar{\theta}_z\right|^2  dz  \notag\\
  &\leq -\frac{1}{2}   \int_0^{2\pi} |\bar{\theta}_z|^2 dz  +  \frac{1}{2}\int_0^{2\pi}   \left|\bar{\theta'_1 w}\right|^2 dz.
  \end{align}

Using the Cauchy-Schwarz inequality and Lemma \ref{lemma2} as well as estimate (\ref{lin-2}), we have
 \begin{align}    \label{inu-9}
  &\int_0^{2\pi}    \left|\bar{\theta'_1 w}\right|^2 dz
  \leq C\int_0^{2\pi} \Big(\int_{[0,2\pi L]^2}  |\theta'_1|^2 dx dy\Big)    \Big(\int_{[0,2\pi L]^2}  |w|^2 dx dy\Big) dz \notag\\
  &\leq C\Big[\sup_{z\in [0,2\pi]}   \Big(\int_{[0,2\pi L]^2}  |w|^2 dx dy\Big) \Big]   \|\theta'_1\|_2^2
  \leq   C \Big(\|w\|_2^2+\|\partial_z^{\frac{2}{3}}w\|_2^2 \Big)   \|\theta'_1\|_2^2  \notag\\
  &\leq C \|\theta'\|_2^2  \|\theta'_1\|_2^2.
  \end{align}

 Then, combining (\ref{inu-8}) and (\ref{inu-9}) gives
 \begin{align}  \label{inu-10}
 -   \int_{\Omega}    w_2 \frac{\partial \bar{\theta}}{\partial z}  \theta' dx dy dz
 \leq     -  2\pi^2 L^2  \int_0^{2\pi} \left|\bar{\theta}_z\right|^2 dz + C  \|\theta'\|_2^2 \|\theta'_1\|_2^2.
    \end{align}

Using (\ref{inu-6}), (\ref{inu-7}) and (\ref{inu-10}), we infer from (\ref{inu-5}) that
\begin{align}   \label{inu-11}
&\frac{d}{dt} \|\theta'\|_2^2 + \|\nabla_h \theta'\|_2^2 +   4 \pi^2 L^2 \int_0^{2\pi} |\bar{\theta}_z|^2 dz    \notag\\
&\leq  C  \|\theta'\|_2^2   \Big(\|\theta_1'\|_2^2 \|\nabla_h \theta_1'\|_2^2  +  \int_0^{2\pi}    |\partial_z {\bar{\theta_1}}|^2 dz + \|\theta'_1\|_2^2+1 \Big),
 \end{align}
for all $t\in [0,T]$.

By virtue of the energy identity (\ref{EI}), we know that
\begin{align}   \label{inu-12}
\frac{1}{2}\|\theta'_1(t)\|_2^2 + \int_0^t   \|\nabla_h \theta'_1(s) \|_2^2 ds
+ 4\pi^2 L^2 \int_0^t   \int_0^{2\pi}    |\partial_z {\bar{\theta_1}(z,s)}|^2 dz ds = \frac{1}{2}\|\theta'_1(0)\|_2^2,
\end{align}
for all $t\in [0,T]$.  Thanks to Gr\"onwall's inequality and estimate (\ref{inu-12}), we derive from (\ref{inu-11}) that
 \begin{align}  \label{Gron}
 \|\theta'(t)\|_2^2     &\leq  \|\theta'(0)\|_2^2 \,  \exp\Big[{C \int_0^t  \Big(\|\theta_1'(s)\|_2^2 \|\nabla_h \theta_1'(s)\|_2^2  +  \int_0^{2\pi}    |\partial_z {\bar{\theta_1}(z,s)}|^2 dz + \|\theta'_1(s)\|_2^2+1\Big) ds  }\Big] \notag\\
&\leq   \|\theta'(0)\|_2^2 \,  \exp\left[ C (\|\theta_1'(0)\|_2^4 +  \|\theta_1'(0)\|_2^2   +  t) \right],
  \end{align}
for all $t\in [0,T]$. As a result, if $\theta_1'(0)=\theta_2'(0)$, i.e., $\theta'(0)=\theta_1'(0)- \theta_2'(0)=0$, then (\ref{Gron}) implies that $\|\theta'(t)\|_2^2 =0$ for all $t\in [0,T]$, namely, $\theta_1'(t)=\theta_2'(t)$ for all $t\in [0,T]$.

In addition, taking advantage of estimate (\ref{lin-1}) and the Poincar\'e inequality (\ref{Poin}), we know that $\|\mathbf u(t)\|_2^2 + \|w(t)\|_2^2 \leq C \|\theta'(t)\|_2^2=0$ for all $t\in [0,T]$, thus $\mathbf u_1(t)=\mathbf u_2(t)$ and $w_1(t)=w_2(t)$ for all $t\in [0,T]$.

Furthermore, we take the scalar product of equation (\ref{inu-4}) and $-\bar{\theta}$ in the duality between $H^{-1}(0,2\pi)$ and $\dot H^1(0,2\pi)$, then
\begin{align*}
\int_0^{2\pi} |\bar{\theta}_z|^2 dz   =    \int_0^{2\pi}  \left[(\bar{\theta' w_2})   +   (\bar{\theta'_1 w})\right] \bar{\theta}_z dz =0,
\;\;  \text{for all}  \;  t\in [0,T],
\end{align*}
since $\theta'(t)=w(t)=0$ for all $t\in [0,T]$. Thus, $\bar{\theta}_z(t)=0$ for all $t\in [0,T]$. That is,
\begin{align}  \label{epz}
\partial_z \bar{\theta_1}(t) =   \partial_z \bar{\theta_2}(t), \;\; \text{for all}\,\, t\in [0,T].
 \end{align}
From (\ref{epz}), we know that $\bar{\theta_1}(z,t) = \bar{\theta_2}(z,t) + C(t)$ for all $z\in [0,2\pi]$, and for all $t\in [0,T]$.
Then, since $\bar{\theta_1}$ and $\bar{\theta_2}$ both have average zero over $[0,2\pi]$, i.e.,
$\int_0^{2\pi} \bar{\theta_1} dz =   \int_0^{2\pi} \bar{\theta_2} dz =0$, it follows that $C(t)=0$ for all $t\in [0,T]$.
Consequently, $\bar{\theta_1}(z,t) = \bar{\theta_2}(z,t)$ for all $z\in [0,2\pi]$, and for all $t\in [0,T]$.

Since both $\theta_1'=\theta_2'$ and $\bar{\theta_1}=\bar{\theta_2}$ are valid,
then $\theta_1= \theta_1'+ \bar{\theta_1}= \theta_2'+ \bar{\theta_2}= \theta_2$ for all $t\in [0,T]$.

This completes the proof for the uniqueness of weak solutions.

Furthermore, by using (\ref{Gron}), it is easy to obtain the continuous dependence on initial data stated in Theorem \ref{thm-weak}.

 \vspace{0.1 in}

\section{Strong solutions}     \label{sec-strong}
In this section, we prove the global well-posedness of strong solutions for system (\ref{inP-1})-(\ref{inP-5}).
The uniqueness of strong solutions follows from the uniqueness of weak solutions. It remains to show the existence of global strong solutions, when the initial value $\theta_0' \in H^1(\Omega)$.

\subsection{Existence of strong solutions}

We use the method of Galerkin approximation. Let us consider the Galerkin system (\ref{inG-1})-(\ref{inG-4}) and perform energy estimate as follows.

\subsubsection{Estimate from linear equations (\ref{inG-1})-(\ref{inG-2})}
Recall the following estimate from subsection \ref{sec-energy}.
\begin{align}
&\|\Delta_h \mathbf u_m\|_2^2  \leq  \text{Ra}^2  \|\theta'_m\|_2^2,    \;\;\;\;   \|\Delta_h w_m\|_2^2\leq  \text{Ra}^2 \|\theta'_m\|_2^2;    \label{sin-1}  \\
&\|\partial_z \mathbf u_m\|_2^2  \leq   \text{Ra}^2 \|\nabla_h \theta'_m\|_2^2,    \;\;\;\;  \|\partial_z w_m\|_2^2 \leq   \text{Ra}^2 \|\nabla_h \theta'_m\|_2^2;    \label{sin-2}   \\
&\|\partial_z^{\frac{2}{3}} \mathbf u_m\|_2^2  \leq   \text{Ra}^2 \|\theta'_m\|_2^2,
\;\;\;\;  \|\partial_z^{\frac{2}{3}} w_m\|_2^2 \leq    \text{Ra}^2 \|\theta'_m\|_2^2.    \label{sin-5}
\end{align}
Corresponding to (\ref{sin-2}), we can also derive
\begin{align}
&\|\nabla_h \partial_z \mathbf u_m\|_2^2   \leq    \text{Ra}^2 \|\Delta_h \theta'_m\|_2^2, \;\;\;\;
\|\nabla_h \partial_z w_m\|_2^2 \leq  \text{Ra}^2    \|\Delta_h \theta'_m\|_2^2;      \label{sin-3}       \\
& \|  \partial_{zz} \mathbf u_m\|_2^2    \leq       \text{Ra}^2 \|\nabla_h \partial_z \theta'_m\|_2^2,
 \;\;\;\;      \|\partial_{zz} w_m\|_2^2 \leq   \text{Ra}^2 \|\nabla_h   \partial_z \theta'_m\|_2^2. \label{sin-4}
\end{align}

\vspace{0.1 in}

\subsubsection{Estimate for $\|\theta'_m\|_2^2$}
This is exactly the same as the energy estimate performed in subsection \ref{sec-energy}. Therefore,  by (\ref{wk-0}), we have
\begin{align}  \label{strong-0}
\frac{1}{2}  \|\theta'_m(t)\|_2^2   + \int_0^t     \|\nabla_h \theta'_m \|_2^2 ds +
4\pi^2 L^2   \int_0^t    \int_0^{2\pi}  |\partial_z \bar{\theta_m}|^2 dz ds  = \frac{1}{2}   \|\theta'_m(0)\|_2^2  \leq   \frac{1}{2}   \|\theta'_0\|_2^2,
\end{align}
for all $t\geq 0$.

\subsubsection{Estimate for $\|\nabla_h \theta'_m\|_2^2$} Taking the $L^2$ inner product of (\ref{inG-3}) with $-\Delta_h \theta'_m$ yields
\begin{align}  \label{in-3}
&\frac{1}{2}\frac{d}{dt} \|\nabla_h \theta'_m\|_2^2 + \|\Delta_h \theta'_m\|_2^2 \notag\\
&= \int_{\Omega} [P_m (\mathbf u_m \cdot \nabla_h \theta'_m )]  \Delta_h \theta'_m dx dy dz
+\int_{\Omega}  [P_m(w_m \frac{\partial \bar{\theta_m}}{\partial z})] \Delta_h \theta'_m  dx dy dz .
\end{align}
Thanks to Lemma \ref{lemma1}, we have
\begin{align}  \label{in-4}
 &\int_{\Omega} [P_m (\mathbf u_m \cdot \nabla_h \theta'_m )]  \Delta_h \theta'_m dx dy dz
 =   \int_{\Omega} (\mathbf u_m \cdot \nabla_h \theta'_m )  \Delta_h \theta'_m dx dy dz   \notag\\
 &\leq C \|\nabla_h \mathbf u_m\|_2^{1/2} \left(\|\mathbf u_m\|_2 + \|\partial_z \mathbf u_m\|_2 \right)^{1/2} \|\nabla_h \theta'_m\|_2^{1/2}  \|\Delta_h \theta'_m\|_2^{3/2}   \notag\\
 &\leq  \frac{1}{4}   \|\Delta_h \theta'_m\|_2^2  +      C\|\nabla_h \mathbf u_m\|_2^2 \left(\|\mathbf u_m\|_2^2 + \|\partial_z \mathbf u_m\|_2^2 \right) \|\nabla_h \theta'_m\|_2^2   \notag\\
 &\leq    \frac{1}{4}   \|\Delta_h \theta'_m\|_2^2 +       C \|\theta'_m\|_2^2  \|\nabla_h \theta'_m\|_2^4  ,
    \end{align}
where we have used estimates (\ref{sin-1})-(\ref{sin-2}).

Next, using the Cauchy-Schwarz inequality, we estimate
\begin{align}   \label{in-5}
&  \int_{\Omega}  [P_m(w_m \frac{\partial \bar{\theta_m}}{\partial z})] \Delta_h \theta'_m  dx dy dz
= \int_0^{2\pi}  \Big(\int_{[0,2\pi L]^2} w_m \Delta_h \theta'_m  dx dy \Big)
\frac{\partial \bar{\theta_m}}{\partial z}   dz  \notag\\
& \leq  \int_0^{2\pi}     \Big( \int_{[0,2\pi L]^2} |w_m|^2 dx dy  \Big)^{1/2}     \Big( \int_{[0,2\pi L]^2} |\Delta_h \theta'_m|^2 dx dy  \Big)^{1/2}  \Big|\frac{\partial \bar{\theta_m}}{\partial z}\Big| dz  \notag\\
&\leq   \Big[\sup_{z\in [0,2\pi]} \Big( \int_{[0,2\pi L]^2} |w_m|^2 dx dy  \Big)^{1/2}  \Big] \|\Delta_h \theta'_m\|_2  \Big(\int_0^{2\pi} \Big|\frac{\partial \bar{\theta_m}}{\partial z}\Big|^2 dz\Big)^{1/2},   \notag\\
&\leq C\Big(\|w_m\|_2+ \|\partial_z^{\frac{2}{3}} w_m\|_2 \Big)
\|\Delta_h \theta'_m\|_2  \Big(\int_0^{2\pi} |\partial_z \bar{\theta_m}|^2 dz \Big)^{1/2} \notag\\
&\leq \frac{1}{4}    \|\Delta_h \theta'_m\|_2^2 + C \Big(\|w_m\|_2^2+ \|\partial_z^{\frac{2}{3}} w_m\|_2^2 \Big)   \Big(\int_0^{2\pi} |\partial_z \bar{\theta_m}|^2 dz\Big)  \notag\\
&\leq \frac{1}{4}    \|\Delta_h \theta'_m\|_2^2 + C   \|\theta'_m\|_2^2   \Big(\int_0^{2\pi} |\partial_z \bar{\theta_m}|^2 dz\Big),
\end{align}
where we have used Lemma \ref{lemma2} and estimates (\ref{sin-1}) and (\ref{sin-5}).

By applying estimates (\ref{in-4}) and (\ref{in-5}) into (\ref{in-3}), we infer
\begin{align}   \label{in-6}
\frac{d}{dt} \|\nabla_h \theta'_m\|_2^2 + \|\Delta_h \theta'_m\|_2^2
& \leq    C\|\theta'_m\|_2^2  \|\nabla_h \theta'_m\|_2^4
+   C   \|\theta'_m\|_2^2
 \Big(\int_0^{2\pi} |\partial_z \bar{\theta_m}|^2 dz\Big).
\end{align}
Applying Gr\"onwall's inequality to (\ref{in-6}) and using estimate (\ref{strong-0}), we obtain
\begin{align}   \label{in-7}
\|\nabla_h \theta'_m (t)\|_2^2
&\leq e^{ C \|\theta'_0\|_2^2   \int_0^t \|\nabla_h \theta'_m\|_2^2 ds  } \Big(   \|\nabla_h \theta'_0\|_2^2    +  C\|\theta'_0\|_2^2  \int_0^t  \int_0^{2\pi} |\partial_z \bar{\theta_m}|^2 dz ds      \Big)   \notag\\
&\leq  e^{C  \|\theta'_0\|_2^4}  (\|\nabla_h \theta'_0\|_2^2    +  C\|\theta'_0\|_2^4 ),   \;\; \text{for all} \;  t\geq 0.
\end{align}
Also, by integrating (\ref{in-6}) over $[0,t]$, and using (\ref{strong-0}) and (\ref{in-7}), one has
\begin{align}  \label{in-77}
\int_0^t    \|\Delta_h \theta'_m(s)\|_2^2 ds   \leq C(\|\nabla_h \theta'_0\|_2^2),   \;\; \text{for all} \;  t\geq 0,
\end{align}
where $C(\|\nabla_h \theta'_0\|_2^2)$ is a constant depending on $\|\nabla_h \theta'_0\|_2^2$ but independent of time.

\subsubsection{Estimate for $\|\partial_z \theta'_m\|_2^2$}
Differentiating (\ref{inG-3}) with respect to $z$ gives
\begin{align}   \label{in-8b}
&\partial_t \partial_z \theta'_m  + P_m(\partial_z \mathbf u_m \cdot \nabla_h \theta'_m) + P_m(\mathbf u_m \cdot \nabla_h \partial_z \theta'_m)   \notag\\
&\hspace{0.2 in} +P_m(\partial_z w_m \partial_z \bar{\theta_m})  + P_m(w_m \partial_{zz} \bar{\theta_m})= \Delta_h \partial_z \theta'_m.
\end{align}

Taking the $L^2$ inner product of (\ref{in-8b}) with $\partial_{z}\theta'_m$, we obtain
\begin{align}   \label{in-8}
&\frac{1}{2} \frac{d}{dt} \|\partial_z \theta'_m\|_2^2 + \|\nabla_h \partial_z \theta'_m\|_2^2 =  -  \int_{\Omega} \left(\partial_z \mathbf u_m \cdot \nabla_h \theta'_m\right)\partial_z \theta'_m dx dy dz    \notag\\
&\hspace{0.2 in} -   \int_{\Omega} \partial_z w_m (\partial_z \bar{\theta_m})  \partial_z \theta'_m dx dy dz
 - \int_{\Omega} w_m (\partial_{zz}\bar{\theta_m})  \partial_z \theta'_m dx dy dz,
\end{align}
where we have used $\int_{\Omega} \left(\mathbf u_m \cdot \nabla_h \partial_z \theta'_m\right)\partial_z \theta'_m dx dy dz=0$ since $\nabla_h \cdot \mathbf u_m=0$.

We evaluate every integral on the right-hand side of (\ref{in-8}). First, by using Lemma \ref{lemma1} as well as estimate (\ref{sin-3})-(\ref{sin-4}), we have
\begin{align}  \label{in-9}
 &\int_{\Omega} |\left(\partial_z \mathbf u_m \cdot \nabla_h \theta'_m\right)\partial_z \theta'_m| dx dy dz \notag\\
 &  \leq C \|\nabla_h \partial_z \mathbf u_m\|_2^{1/2} \|\partial_{zz}\mathbf u_m\|_2^{1/2}
   \|\nabla_h \theta'_m\|_2^{1/2}    \|\Delta_h \theta'_m\|_2^{1/2}    \|\partial_z \theta'_m\|_2 \notag\\
   &\leq C \|\Delta_h \theta'_m\|_2 \|\nabla_h \partial_z \theta'_m\|_2^{1/2}      \|\nabla_h \theta'_m\|_2^{1/2}    \|\partial_z \theta'_m\|_2 \notag\\
   & \leq \frac{1}{4}   \|\nabla_h \partial_z \theta'_m\|_2^2 +     \|\nabla_h \theta'_m\|_2^2
   +    C   \|\Delta_h \theta'_m\|_2^2   \|\partial_z \theta'_m\|_2^2.
 \end{align}

Applying the Cauchy-Schwarz inequality yields
\begin{align}    \label{in-11}
&\int_{\Omega} |\partial_z w_m (\partial_z \bar{\theta_m}) \partial_z  \theta'_m| dx dy dz   \notag\\
&\leq \int_0^{2\pi} \Big(\int_{[0,2\pi L]^2} |\partial_z w_m|^2 dx dy\Big)^{1/2}    \Big(\int_{[0,2\pi L]^2} |\partial_z\theta'_m|^2 dx dy\Big)^{1/2} |\partial_z \bar{\theta_m}| dz  \notag\\
&\leq  \Big[\sup_{z\in [0,2\pi]}   \Big(\int_{[0,2\pi L]^2} |\partial_z w_m|^2 dx dy\Big)^{1/2} \Big]  \|\partial_z \theta'_m\|_2  \Big(\int_0^{2\pi} |\partial_z \bar{\theta_m}|^2 dz\Big)^{1/2}, \notag\\
&\leq C  \|\partial_{zz}w_m\|_2
\|\partial_z \theta'_m\|_2  \Big(\int_0^{2\pi} |\partial_z \bar{\theta_m}|^2 dz \Big)^{1/2}  \notag\\
&\leq C \|\nabla_h \partial_z \theta'_m\|_2    \|\partial_z \theta'_m\|_2  \Big(\int_0^{2\pi} |\partial_z \bar{\theta_m}|^2 dz\Big)^{1/2}  \notag\\
&\leq \frac{1}{4}    \|\nabla_h \partial_z \theta'_m\|_2^2  + C  \|\partial_z \theta'_m\|_2^2  \Big(\int_0^{2\pi} |\partial_z \bar{\theta_m}|^2 dz\Big),
\end{align}
where we have used Lemma \ref{lemma2} and estimate (\ref{sin-4}).

Now we estimate the third integral on the right-hand side of (\ref{in-8}). Since $\partial_z (\bar{\theta'_m w_m})= \partial_{zz}\bar{\theta_m}$ due to (\ref{inG-4}), we have
\begin{align}   \label{in-12}
&- \frac{1}{4\pi^2 L^2} \int_{\Omega} w_m (\partial_{zz}\bar{\theta_m})  \partial_z \theta'_m dx dy dz
  = - \frac{1}{4 \pi^2 L^2}  \int_{\Omega} w_m [\partial_z (\bar{\theta'_m w_m})] \partial_z \theta'_m dx dy dz   \notag\\
&= - \int_0^{2\pi}    (\bar{w_m \partial_z\theta'_m})   [\partial_z (\bar{ \theta'_m w_m})] dz   \notag\\
&=  -  \int_0^{2\pi}    (\bar{w_m \partial_z \theta'_m})  \left[ \bar{(\partial_z \theta'_m) w_m}
+ \bar{\theta'_m (\partial_z w_m)} \right]  dz     \notag\\
&= - \int_0^{2\pi}   (\bar{w_m \partial_z \theta'_m})^2 dz
-  \int_0^{2\pi}    (\bar{w_m \partial_z \theta'_m})   (\bar{\theta'_m \partial_z w_m}) dz   \notag\\
&\leq -\frac{1}{2}  \int_0^{2\pi}   (\bar{w_m \partial_z \theta'_m})^2 dz
+  \frac{1}{2}  \int_0^{2\pi}   (\bar{\theta'_m \partial_z w_m})^2 dz,
\end{align}
where the last inequality is due to the Cauchy-Schwarz inequality and the Young's inequality.

By using the Cauchy-Schwarz inequality, and Lemma \ref{lemma2} as well as estimate (\ref{sin-2}), then
 \begin{align}  \label{in-13}
 &\int_0^{2\pi}   (\bar{\theta'_m \partial_z w_m})^2 dz
 \leq  C\int_0^{2\pi} \Big(\int_{[0,2\pi L]^2} |\theta'_m|^2 dx dy \Big)   \Big(\int_{[0,2\pi L]^2} |\partial_z w_m|^2 dx dy\Big) dz \notag\\
 &\leq C \Big[\sup_{z\in [0,2\pi]}   \Big(\int_{[0,2\pi L]^2} |\theta'_m|^2 dx dy \Big) \Big]  \|\partial_z w_m\|_2^2  \leq C   \left(\|\theta'_m\|_2^2 +  \|\partial_z \theta'_m\|_2^2 \right)  \|\nabla_h \theta'_m\|_2^2.
  \end{align}
 Substituting (\ref{in-13}) into (\ref{in-12}) gives
 \begin{align}   \label{in-14}
&- \frac{1}{4 \pi^2 L^2} \int_{\Omega} w_m (\partial_{zz}\bar{\theta_m})  \partial_z \theta'_m dx dy dz   \notag\\
&\leq   -\frac{1}{2}  \int_0^{2\pi}   (\bar{w_m \partial_z\theta'_m})^2 dz
+   C   \left(\|\theta'_m\|_2^2 +  \|\partial_z \theta'_m\|_2^2 \right)  \|\nabla_h \theta'_m\|_2^2.
 \end{align}

Using (\ref{in-9}), (\ref{in-11}) and (\ref{in-14}), we infer from (\ref{in-8}) that
\begin{align}  \label{in-15}
&\frac{d}{dt} \|\partial_z \theta'_m\|_2^2 + \|\nabla_h \partial_z \theta'_m\|_2^2   \notag\\
&\leq    C  \|\partial_z \theta'_m\|_2^2 \left[ \|\Delta_h \theta'_m\|_2^2   + \int_0^{2\pi} |\partial_z \bar{\theta_m}|^2 dz \right]
+C (\|\theta'_m\|_2^2+1)   \|\nabla_h \theta'_m\|_2^2.
 \end{align}
Due to (\ref{strong-0}) and (\ref{in-7})-(\ref{in-77}), we can use Gr\"onwall's inequality on (\ref{in-15}) to obtain
\begin{align}  \label{in-H1z}
\|\partial_z \theta'_m(t)\|_2^2
&\leq e^{\int_0^t   (\|\Delta_h \theta'_m\|_2^2   + \int_0^{2\pi} |\partial_z \bar{\theta_m}|^2 dz) ds  }
\Big(\|\partial_z \theta_0'\|_2^2 +  C (\|\theta'_0\|_2^2+1)      \int_0^t   \|\nabla_h \theta'_m\|_2^2 ds \Big)  \notag\\
&\leq C(\|\theta'_0\|_{H^1(\Omega)}),  \;\;  \text{for all} \; t\geq 0,
\end{align}
where $C(\|\theta'_0\|_{H^1(\Omega)})$ is a constant depending on $\|\theta'_0\|_{H^1(\Omega)}$ but independent of $t$.
In addition, integrating (\ref{in-15}) over $[0,t]$ gives
\begin{align}       \label{in-H1zz}
\int_0^t   \|\nabla_h \partial_z  \theta'_m\|_2^2  ds   \leq    C(\|\theta'_0\|_{H^1(\Omega)}),      \;\;  \text{for all} \; t\geq 0.
\end{align}

On account of (\ref{strong-0}), (\ref{in-7}) and (\ref{in-H1z}), we conclude that
\begin{align}   \label{in-H1ub}
\|\theta'_m(t)\|_{H^1(\Omega)}^2 \leq   C(\|\theta'_0\|_{H^1(\Omega)}), \;\; \text{for all} \;t\geq 0,
\end{align}
where $C(\|\theta'_0\|_{H^1(\Omega)})$ is a constant depending on $\|\theta'_0\|_{H^1(\Omega)}$ but independent of $t$.

\subsubsection{Estimate for $\int_0^{2\pi}\left|\partial_{zz}\bar{\theta_m}\right|^2 dz$}
Since $ \frac{\partial^2 \bar{\theta_m}}{\partial z^2}  =  \frac{\partial(\bar{\theta'_m w_m})}{\partial z} $, we have
\begin{align}    \label{pa-66}
&\int_0^{2\pi}  \Big|\frac{\partial^2 \bar{\theta_m}}{\partial z^2} \Big|^2 dz
 =    \int_0^{2\pi} \Big|\frac{\partial(\bar{\theta'_m  w_m})}{\partial z}\Big|^2 dz
=  \int_0^{2\pi}  \Big|   \bar{\frac{\partial \theta'_m}{\partial z} w_m}   +  \bar{\theta'_m \frac{\partial w_m}{\partial z} }     \Big|^2  dz
\notag\\
&\leq C \int_0^{2\pi} \Big|    \int_{[0,2\pi L]^2}  \frac{\partial \theta'_m}{\partial z} w_m dx dy     \Big|^2  dz
+ C \int_0^{2\pi} \Big|  \int_{[0,2\pi L]^2}  \theta'_m   \frac{\partial w_m}{\partial z} dx dy \Big|^2 dz  \notag\\
&\leq    C \int_0^{2\pi} \Big(\int_{[0,2\pi L]^2}  \Big|\frac{\partial \theta'_m}{\partial z}\Big|^2 dx dy\Big)
\Big(\int_{[0,2\pi L]^2}  | w_m |^2 dx dy\Big)   dz           \notag\\
& \hspace{0.2 in}  +   C \int_0^{2\pi} \Big(\int_{[0,2\pi L]^2}  |\theta'_m|^2 dx dy\Big)
\Big(\int_{[0,2\pi L]^2}  \Big|   \frac{\partial w_m}{\partial z} \Big|^2 dx dy\Big)   dz           \notag\\
&\leq C \Big(\sup_{z\in [0,2\pi]} \int_{[0,2\pi L]^2} |w_m|^2 dx dy\Big) \|\partial_z \theta'_m\|_2^2 +
C \Big(\sup_{z\in [0,2\pi]} \int_{[0,2\pi L]^2} |\theta'_m|^2 dx dy\Big) \|\partial_z w_m\|_2^2  \notag\\
&\leq C  (\|w_m\|_2^2+ \|\partial_z w_m\|_2^2)  \|\partial_z \theta'_m\|_2^2
+  C (\|\theta'_m\|_2^2+ \|\partial_z \theta'_m\|_2^2)  \|\partial_z w_m\|_2^2   \notag\\
&\leq C \|\nabla_h \theta'_m\|_2^2  \|\partial_z \theta'_m\|_2^2 +  C (\|\theta'_m\|_2^2+ \|\partial_z \theta'_m\|_2^2)  \|\nabla_h \theta'_m\|_2^2,
\end{align}
for all $t\geq 0$, where we have used Lemma \ref{lemma2} and estimates (\ref{sin-1}) and (\ref{sin-2}).

Applying uniform bound (\ref{in-H1ub}) to estimate (\ref{pa-66}) yield
\begin{align}    \label{barzz-1}
\int_0^{2\pi}  \Big|\frac{\partial^2 \bar{\theta_m}}{\partial z^2}(z,t) \Big|^2 dz  \leq   C(\|\theta'_0\|_{H^1}),  \;\; \text{for all}   \; t\geq 0,
\end{align}
where $C(\|\theta'_0\|_{H^1(\Omega)})$ is a constant depending on $\|\theta'_0\|_{H^1(\Omega)}$ but independent of $t$.

\vspace{0.1 in}

\subsubsection{Passage to the limit}
Let $T>0$. In order to pass to the limit for nonlinear terms in the Galerkin system as $m\rightarrow \infty$, we shall show that $\partial_t \theta'_m$ is uniformly bounded in $L^2(\Omega \times (0,T))$.

In fact, for any function $\varphi \in L^{4/3}(0,T;L^2(\Omega))$, using Lemma \ref{lemma1}, one has
\begin{align}  \label{non-1}
&\int_0^T \int_{\Omega} [P_m(\mathbf u_m \cdot \nabla_h \theta'_m)] \varphi dx dy dz dt   \notag\\
&\leq C\int_0^T  \|\nabla_h \mathbf u_m\|^{1/2} \left( \|\mathbf u_m\|_2 + \|\partial_z \mathbf u_m\|_2 \right)^{1/2} \|\nabla_h \theta'_m\|_2^{1/2}  \|\Delta_h \theta'_m\|_2^{1/2} \|\varphi\|_2  dt \notag\\
&\leq  C\int_0^T  \|\theta_m'\|^{1/2}  \|\nabla_h \theta'_m\|_2 \|\Delta_h \theta'_m\|_2^{1/2} \|\varphi\|_2  dt   \notag\\
&\leq C(\|\nabla_h \theta'_0\|_2)    \int_0^T         \|\Delta_h \theta'_m\|_2^{1/2} \|\varphi\|_2  dt \notag\\
&\leq C(\|\nabla_h \theta'_0\|_2)   \left(\int_0^T    \|\Delta_h \theta'_m\|_2^2   dt\right)^{1/4}
\left(\int_0^T  \|\varphi\|_2^{4/3}  dt \right)^{3/4}  \notag\\
&\leq C(\|\nabla_h \theta'_0\|_2)   \|\varphi\|_{L^{4/3}(0,T;L^2(\Omega))},
\end{align}
where we have used estimate (\ref{sin-1})-(\ref{sin-2}) and (\ref{in-7})-(\ref{in-77}).
Here, $C(\|\nabla_h \theta'_0\|_2)$ is a constant depending on $\|\nabla_h \theta'_0\|_2$ but independent of $t$.

We infer from (\ref{non-1}) that
\begin{align}  \label{non-2}
\|P_m(\mathbf u_m \cdot \nabla_h \theta'_m)\|_{L^4(0,T;L^2(\Omega))} \leq   C(\|\nabla_h \theta'_0\|_2).
\end{align}

Then, using (\ref{inG-3}), (\ref{wk-15}), (\ref{in-77}) and (\ref{non-2}), we obtain
\begin{align}    \label{non-6}
\|\partial_t \theta'_m\|_{L^2(\Omega \times (0,T))} \leq  C(\|\nabla_h \theta'_0\|_2)   .
\end{align}
Therefore, on a subsequence, we have
\begin{align}  \label{non-7}
\partial_t \theta_m' \rightarrow \partial_t \theta'  \;\; \text{weakly in} \; L^2(\Omega \times (0,T)),   \; \text{as}  \;  m\rightarrow \infty.
\end{align}

In addition, due to uniform bounds (\ref{in-H1ub}) and (\ref{non-6}), and thanks to the Aubin's compactness theorem (see, e.g., \cite{Temam}), the following strong convergence holds for a subsequence of $\{\theta'_m\}$:
\begin{align}   \label{stron-1}
\theta_m' \rightarrow \theta'  \;\; \text{in} \; L^2(\Omega \times (0,T)),   \; \text{as} \;  m\rightarrow \infty.
\end{align}

Also, using (\ref{FourG-3}), we see that
\begin{align}  \label{stron-2}
\|\mathbf u_m -\mathbf u_n\|_2^2 \leq  C\|\theta'_m - \theta'_n\|_2^2,  \;\; \text{for any} \; n,m \in \mathbb N.
\end{align}
We infer from (\ref{stron-1})-(\ref{stron-2}) that $\{\mathbf u_m\}$ is a Cauchy sequence in $L^2(\Omega \times (0,T))$, then
\begin{align}   \label{stron-3}
\mathbf u_m \rightarrow \mathbf u    \;\; \text{in} \;\;L^2(\Omega \times (0,T)),   \; \text{as} \; m\rightarrow \infty.
\end{align}
Analogously, we deduce
\begin{align}   \label{stron-4}
w_m \rightarrow w      \;\; \text{in} \;\;L^2(\Omega \times (0,T)),    \; \text{as}   \; m\rightarrow \infty.
\end{align}

Hence, we can use the same argument as in subsection \ref{sec-pass} to pass to the limit $m\rightarrow \infty$ for the Galerkin system (\ref{inG-1})-(\ref{inG-4}) to derive
\begin{align}   \label{pa-5}
\frac{\partial \theta'}{\partial t}  + \mathbf u \cdot \nabla_h \theta' +  w \frac{\partial \bar{\theta}}{\partial z} -\Delta_h \theta'=0  \;\; \text{in} \;\;  L^2(\Omega \times (0,T)).
\end{align}

Next, we shall pass to the limit for equation (\ref{inG-4}) in the Galerkin system as $m\rightarrow \infty$. By virtue of (\ref{barzz-1}),
$\int_0^{2\pi}  \left|\frac{\partial^2 \bar{\theta_m}  }{\partial z^2}(z,t) \right|^2 dz  \leq   C(\|\theta'_0\|_{H^1})$, for all $t\geq 0$, namely, the sequence $\{\frac{\partial^2 \bar{\theta_m}}{\partial z^2}\}$ is uniformly bounded in $L^{\infty}(0,T;L^2(0,2\pi))$, which implies, on a subsequence,
\begin{align}  \label{pa-7}
\frac{\partial^2 \bar{\theta_m}}{\partial z^2}  \rightarrow   \frac{\partial^2 \bar{\theta}}{\partial z^2}  \;\;  \text{weakly$^*$ in}  \;L^{\infty}(0,T;L^2(0,2\pi)),
\; \text{as} \; m\rightarrow \infty.
\end{align}

In order to pass to the limit for the nonlinear term $\frac{\partial(\bar{\theta_m'  w_m})}{\partial z}$ in equation (\ref{inG-4}), we take a test function $\xi=\sum_{|j|<N}  \hat{\xi}(t,j) e^{ij z}$ where Fourier coefficients $\hat{\xi}(t,j)$, $|j|\leq N$, are continuous in $t$. Consider
\begin{align}   \label{pa-8}
& \left|\int_0^T \int_0^{2\pi} \left(\frac{\partial (\bar{\theta_m'  w_m})}{\partial z} - \frac{\partial (\bar{\theta'  w})}{\partial z}\right) \xi   dz dt   \right|
=   \left|\int_0^T \int_0^{2\pi} (\bar{\theta_m'  w_m}-   \bar{\theta'  w})   \xi_z  dz dt  \right| \notag\\
&\leq \left|\int_0^T \int_0^{2\pi} (\bar{(\theta_m'-\theta')  w_m}) \xi_z  dz dt   \right|+  \left|   \int_0^T \int_0^{2\pi} (\bar{\theta'  (w_m-w)}) \xi_z  dz dt  \right|   \notag\\
&\leq     C\|\xi_z\|_{L^{\infty}((0,2\pi)\times (0,T))}     \int_0^T \int_{\Omega}   \left(\left|(\theta_m'-\theta')  w_m\right|  +    \left|\theta'  (w_m-w)\right| \right) dx dy dz dt    \notag\\
&\leq  C\|\xi_z\|_{L^{\infty}((0,2\pi)\times (0,T))} \big(\|\theta'_m  -  \theta'\|_{L^2(\Omega \times (0,T))}   \|w_m\|_{L^2(\Omega \times (0,T))}     \notag\\
& \hspace{1.5 in}+ \|w_m -  w \|_{L^2(\Omega \times (0,T))}   \|\theta'\|_{L^2(\Omega \times (0,T))}\big),
\end{align}
where we have used the Cauchy-Schwarz inequality.

Then, since $\|w_m\|_2 \leq C\|\theta'_0\|$ for all $m$ on $[0,T]$ and since $\theta'_m \rightarrow \theta'$,  $w_m \rightarrow w$ in $L^2(\Omega \times (0,T))$, we can let $m\rightarrow \infty$ in (\ref{pa-8}) to obtain
\begin{align}  \label{pa-9}
\lim_{m\rightarrow \infty}\int_0^T \int_0^{2\pi} \frac{\partial (\bar{\theta_m'  w_m})}{\partial z}  \xi  dz dt
= \int_0^T \int_0^{2\pi}   \frac{\partial (\bar{\theta'  w})}{\partial z}  \xi   dz dt.
\end{align}

On account of (\ref{pa-7}) and (\ref{pa-9}), we pass to the limit for the equation (\ref{inG-4}) in the Galerkin approximate system, and arrive at
\begin{align}  \label{pa-10}
\int_0^T \int_0^{2\pi} \left(\frac{\partial (\bar{\theta'  w})}{\partial z} -   \frac{\partial^2 \bar{\theta}}{\partial z^2}\right) \xi  dz dt =0,
\end{align}
for any test function $\xi=\sum_{|j|<N}  \hat{\xi}(t,j) e^{ij z}$ with continuous Fourier coefficients.

Similarly to (\ref{pa-66}), we derive that
$\frac{\partial (\bar{\theta'  w})}{\partial z}  \in  L^{\infty}(0,T;L^2(0,2\pi))$. Also
$\frac{\partial^2 \bar{\theta}}{\partial z^2}   \in  L^{\infty}(0,T;L^2(0,2\pi))$ from (\ref{pa-7}). Hence,
$\frac{\partial (\bar{\theta'  w})}{\partial z}  -    \frac{\partial^2 \bar{\theta}}{\partial z^2}  \in    L^{\infty}(0,T;L^2(0,2\pi))$.
Then, we infer from (\ref{pa-10}) that
\begin{align*}
\frac{\partial (\bar{\theta'  w})}{\partial z}  =   \frac{\partial^2 \bar{\theta}}{\partial z^2}  \;\;\text{in} \;\;  L^{\infty}(0,T;L^2(0,2\pi)).
\end{align*}
This completes the proof for the existence of a global strong solution in the sense of Definition \ref{def-solution}.

\vspace{0.2 in}

\noindent {\bf Acknowledgment.} The work of E.S.T. was  supported in part by the Einstein Stiftung/Foundation - Berlin, through the Einstein Visiting Fellow Program.

\vspace{0.1 in}


\end{document}